\documentclass[11pt,bezier]{article}
\usepackage{amsmath,amssymb,amsfonts,euscript}
\usepackage{cite}
\usepackage[bookmarks]{hyperref}
\renewcommand{\baselinestretch}{1.1}
\usepackage{tikz}
\usepackage{graphicx,url}
\usepackage{scalerel,amssymb,graphicx}
\def\msquare{\mathord{\scalebox{0.5}[0.5]{\scalerel*{\Box}{\strut}}}}
\usepackage{tikz} 
\usetikzlibrary{arrows.meta}
\usetikzlibrary{arrows}
\usepackage{fixltx2e}
\usepackage{amsmath}
\usepackage[utf8]{inputenc}  

\usepackage{verbatim}
\usetikzlibrary{arrows,shapes}

\newtheorem{prethm}{{\bf Theorem}}
\renewcommand{\theprethm}{{\arabic{prethm}}}
\newenvironment{thm}{\begin{prethm}{\hspace{-0.5
               em}{\bf.}}}{\end{prethm}}

               \newtheorem{prelem}[prethm]{Lemma}

\newtheorem{precor}[prethm]{Corollary}

\newtheorem{preconj}[prethm]{Conjecture}
\newenvironment{conj}{\begin{preconj}{\hspace{-0.5
               em}{\bf.}}}{\end{preconj}}

\newtheorem{preremark}{{\bf Remark}}
\renewcommand{\thepreremark}{{\arabic{preremark}}}

\newtheorem{prepro}[prethm]{Proposition}

\newtheorem{preexample}{{\bf Example}}
\renewcommand{\thepreexample}{{\arabic{preexample}}}

\newtheorem{preproof}{{\bf Proof.}}
\renewcommand{\thepreproof}{}
\newenvironment{proof}[1]{\begin{preproof}{\rm
               #1}\hfill{$\Box$}}{\end{preproof}}

\newcommand{\Deg}{\mbox{deg}\,}

\title{\large \bf Star Coloring of the Cartesian Product of Cycles
\thanks
{{\it Key Words}: Vertex coloring, Star coloring, Cartesian products.}
\thanks {2010{ \it Mathematics Subject Classification}: 05C15, 05C38.
 }}

\author{{\normalsize
{\sc S. Akbari${}^{ \mathsf{a}}$},\,
 {\sc M. Chavooshi${}^{\mathsf{b}}$},\,
 {\sc M. Ghanbari${}^{\mathsf{a}}$},\,
 {\sc S. Taghian${}^{\mathsf{a}}$}
 }\vspace{3mm}
\\ {\footnotesize{}}\\{\footnotesize{${}^{\mathsf{a}}$\it
Department of
Mathematical Sciences, Sharif University of Technology, Tehran,
Iran
}}
{\footnotesize{}}\\{\footnotesize{${}^{\mathsf{b}}$\it
Department of
Mathematics, University of Houston, Houston, Texas,
USA}}
\thanks{{\it E-mail addresses}:  $\mathsf{s\_akbari@sharif.edu}$,
$\mathsf{malich@math.uh.edu}$,
$\mathsf{marghanbari@gmail.com}$,
$\mathsf{shadi.taghian@gmail.com}$. } }

\date{}
\begin{document}

\maketitle

\begin{abstract}
{\small
A proper vertex coloring of a graph $G$ is called a \textit{star coloring} if every two color classes induce a forest whose each component is a star, which means there is no bicolored $P_4$ in $G$. In this paper, we show that the Cartesian product of any two cycles, except $C_3 \square C_3$ and $C_3 \square C_5$, has a $5$-star coloring.
}
\end{abstract}


\section{Introduction}

\hspace{3mm}
All graphs considered in this paper are simple which means with no loops or multiple edges. Let $G$ be a graph. We denote the vertex set and the edge set of $G$ by $V(G)$ and $E(G)$, respectively.  In this paper $P_n$ and $C_n$ denote the path and the cycle with $n$ vertices, respectively.
The Cartesian product of two graphs $G_1$ and $G_2$ is a graph denoted by $G_1\square G_2$, with the vertex set $V(G_1\square G_2)=V(G_1) \times V(G_2)$ and $u = (u_1,u_2)$ is adjacent to $v = (v_1,v_2)$ whenever either ($u_1 = v_1$ and $u_2$ is adjacent to  $v_2$) or ($u_2=v_2$ and $u_1$ is adjacent to $v_1$). 

  A {\it proper coloring} of a graph $G$ is an assignment of $k$ colors to the vertices of $G$ such that no two adjacent vertices have the same color.
    A {\it star coloring} of a graph $G$ is a proper vertex coloring of $G$ such that no $P_4$ is bicolored. Indeed, a star coloring of a graph $G$ is a proper vertex coloring of $G$ such that
the union of every two color classes induces a forest whose each component is a star.
The {\it star chromatic number} of a graph $G$ is the minimum
number of colors which are necessary to color the vertices of $G$ such that $G$ has a star coloring and is denoted by $\chi_s(G)$. 
This definition was first introduced by Gr\"{u}nbaum in \cite{2}.  Later, it was well studied and has been widely investigated \cite{3,7}. An {\it acyclic coloring} is a proper vertex coloring such that the union of any two color classes induces a forest. If this coloring is not necessarily proper, then it is called the {\it acyclic improper coloring} of graphs, see \cite{acyclicimproper}. Star colorings are a strengthening of acyclic coloring, which is a proper vertex coloring in which every two color classes induce a forest (and there is no bicolored cycle). 

Albertson et al. \cite{9} showed that the star coloring problem is NP-complete even restricted to planar bipartite graphs. Fertin et. al. \cite{4} determined the exact value of the star chromatic number of different families of graphs such as trees, cycles, complete bipartite graphs, outer planar graphs and $2$-dimensional grids. They also provided bounds for the star chromatic numbers of other families of graphs, such as planar graphs, hypercubes, $d$-dimensional grids ($d\ge3$), graphs with bounded treewidth, and cubic graphs. However, the star coloring of planar graphs has attracted lots of attention. Gr\"{u}nbaum in \cite{2} found some relations between star coloring and acyclic coloring and provided an upper bound $2304$ for the star chromatic number of a planar graph. Later, this upper bound was reduced to $80$ by \cite{8} and to $30$ by \cite{6}. However, Albertson et. al. \cite{9} improved this bound to $20$. Regarding the lower bound, Fertin et. al. \cite{4} gave a planar graph with its star chromatic number $6$. Moreover, in \cite{2018}, the upper bound for the star chromatic number of graphs in terms of the maximum number of vertices that are pairwise adjacent in $G$ is given. 

It is not hard to see  that for every natural number $n\geq 3$,
$
\chi_s(P_n) = 3
$
and
		$$
		\chi_s(C_n) =
		\begin{cases}
		4       & \quad n = 5\\
		3  & \quad \text{otherwise}.
		\end{cases}
		$$
	
It was proved that for natural numbers $m$ and $n$, $min \{ m,n \} \geq 4$, $\chi_s(P_m \square P_n) = 5$, see \cite{2016}. Moreover, in \cite{2016} it is proved that for every two natural numbers  $m,n\geq3$, $4\leq\chi_s(P_m \square C_n) \leq 5$.
The following conjecture was proposed in \cite{3}.

	\begin{conj}\label{conj}
Let $m,n \geq 3$ be two natural numbers. Then
		$$
		\chi_s(C_m \square C_n) =
		\begin{cases}
		6       & \quad m=n=3 \quad or \quad m=3,n = 5\\
		5  & \quad \text{otherwise}.
		\end{cases}
		$$
	\end{conj}

Recently, in \cite{2016}, this conjecture was proved for $m = n = 3$ and also for every two natural numbers $m,n\geq                 30$.  In this paper, we prove this conjecture for all cases. 
 
\section{Results}

In this section, we would like to show that the Cartesian product of any two cycles, except $C_3 \square C_3$ and $C_3 \square C_5$, has a $5$-star coloring.
It was proved that for every two graphs $G$ and $H$, $\chi_s(G \square H) \leq \chi_s(G) \chi_s(H)$, see \cite{4}. 
The following results were proved on star coloring of the Cartesian product of two graphs.

\begin{thm}\rm{\cite{4}} For $i=2,3$, $\chi_s(P_2\square P_i) = i+1$ and for every two natural numbers $m$ and $n$, $  \min \{ m,n \}\ge 4$,
  $\chi_s(P_2\square P_n) = \chi_s(P_3\square  P_n) = \chi_s(P_3\square  P_3) = 4$ and
   $\chi_s(P_m \square P_n) = 5$. 
\end{thm}

\begin{thm}\rm{\cite{2016}}
  Let $n \ge  3$ be a natural number, then $$
		\chi_s(P_3 \square C_n) =
		\begin{cases}
		4       & \,n\equiv\,0\,(mod\,\,2)\\
		5  & \,n\equiv\,1\,(mod\,\,2).
		\end{cases}
		$$
   
\end{thm}

\begin{thm}\rm{\cite{fertin}}\label{fertin}
 For natural numbers  $m,n\ge 3$, $5\le \chi_s(C_m \square C_n)\le 6$.
\end{thm}

\begin{thm}\rm{\cite{2016}}
 For natural numbers $m,n\ge 4$, $\chi_s(P_m \square C_n)=5$.
\end{thm}

\begin{thm}\rm{\cite{2016}}
 For natural numbers  $m,n\ge 30$, $\chi_s(C_m \square C_n)=5$.
\end{thm}

We start the paper with the following result.

\begin{thm}
{$\chi_s(C_3\square C_5)=6$.}
\end{thm}

\begin{proof}         
{Let $G=C_3\square C_5$. Clearly, $G$ is a grid with $3$ rows and $5$ columns. To the contrary and by Theorem \ref{fertin},  assume that $G$ has a star coloring by the colors $\{a,b,c,d,e\}$. First, we want to show that each color can be appeared at most $3$ times in the star coloring of $G$. 
To see this, first note that since every two vertices on each column of $G$ are adjacent, thus each color can be appeared at most once in each column of $G$ which implies that each color is appeared at most $5$ times in the star coloring of $G$. Now, with no loss of generality assume that the color $a$ is appeared $5$ times which means that the color $a$  is appeared in each column. Thus there are $10$ remaining vertices and $4$ available colors. Then by the pigeonhole principle there exists a color, say $b$, is appeared at least $3$ times. Since the color $b$ should appear in distinct columns, we conclude that  there are at least three columns containing the colors $\{a,b\}$ which implies that there are two consecutive columns containing the  colors $\{a,b\}$, which contradicts having no $2$-colored $P_4$. Thus, each color can be appeared at most $4$ times in the star coloring of $G$. Assume that the color $a$  is appeared $4$ times. Note that not more than one color can appear $4$ times, otherwise there are two consecutive columns containing the same pair of colors, a contradiction.
Hence, there are $11$ remaining vertices and $4$ available colors which the pigeonhole principle implies that one color, say $b$, is appeared exactly $3$ times. Now, there are $8$ remaining vertices and $3$ colors left which the pigeonhole principle implies that one color, say $c$, is appeared $3$ times, too. Next there are $5$ remaining vertices and $2$ colors left which implies that two colors, $d$ and $e$, are appeared $3$ and $2$ times, respectively. Since each color appears in each column at most once, one may suppose that the color $a$
is appeared in the first four columns. If each of colors $\{b,c,d\}$ are appeared three times in the columns which $a$ is appeared, then there are two consecutive columns containing the same pair of colors, a contradiction. Thus each of the colors $\{b,c,d\}$ is appeared exactly twice in the first four columns. Therefore the fifth column consists of the colors $\{b,c,d\}$. Clearly, if the fourth column contains two colors of $\{b,c,d\}$,  then there are two consecutive columns containing the same pair of colors, a contradiction. Thus with no loss of generality we can assume that the fourth  column contains the colors $\{a,b,e\}$. This implies that the colors $\{c,d\}$ should be appeared in two columns of the first three columns of $G$ including the color $a$. This forces the first and the third columns of $G$ consist of $\{a,c,d\}$, otherwise there are two consecutive columns containing $\{a,x\}$, where $x\in \{c,d\}$, a contradiction. Therefore the first and the fifth columns contain the same pair of colors, a contradiction. This yields that each color is appeared exactly three times and since there are $5$ columns, for each color there are two consecutive columns containing this color.  With no loss of generality, assume that the color $a$ is appeared in the first and second columns of $G$. Now, if  each colors of $\{b,c,d,e\}$ is appeared twice in these two columns, then we get a contradiction. This concludes that each of
the colors $\{b,c,d,e\}$ is appeared twice in the columns third, fourth and fifth.
 To choose two columns among these three columns we have ${3\choose 2}=3$ choices and also we have $4$ available colors $\{b,c,d,e\}$, by the pigeonhole principle we conclude that there are two consecutive columns containing the same pair of colors, a contradiction. The proof is complete.
In the following,  a $6$-star coloring of $C_3\square C_5$ is given.

\begin{center}

	\begin{tikzpicture}
	\begin{scope}[every node/.style={circle,thick,draw,scale=0.5}]
	\node (A) at (0,0) {\textbf{5}};
	\node (B) at (0,1) {\textbf{3}};
	\node (C) at (0,2) {\textbf{1}};
	\node (F) at (1,0) {\textbf{3}};
	\node (G) at (1,1) {\textbf{4}};
	\node (H) at (1,2) {\textbf{2}};
	\node (K) at (2,0) {\textbf{1}};
	\node (L) at (2,1) {\textbf{5}};
	\node (M) at (2,2) {\textbf{3}};
	\node (P) at (3,0) {\textbf{6}};
	\node (Q) at (3,1) {\textbf{1}};
	\node (R) at (3,2) {\textbf{4}};
	\node (S) at (4,0) {\textbf{3}};
	\node (T) at (4,1) {\textbf{2}};
	\node (U) at (4,2) {\textbf{6}};
	\end{scope}
	
	\begin{scope}[>={[black]},
	every node/.style={fill=white,circle},
	every edge/.style={draw=black}]
	\path [->] (A) edge (B);
	\path [->] (B) edge (C);
	\path [->] (F) edge (G);
	\path [->] (G) edge (H);
	\path [->] (K) edge (L);
	\path [->] (L) edge (M);
	\path [->] (P) edge (Q);
	\path [->] (Q) edge (R);
	\path [->] (A) edge (F);
	\path [->] (F) edge (K);
	\path [->] (K) edge (P);
	\path [->] (B) edge (G);
	\path [->] (G) edge (L);
	\path [->] (L) edge (Q);
	\path [->] (C) edge (H);
	\path [->] (H) edge (M);
	\path [->] (M) edge (R);
	\path [->] (S) edge (T);
	\path [->] (T) edge (U);
	\path [->] (P) edge (S);
	\path [->] (Q) edge (T);
	\path [->] (R) edge (U);
	\path [->] (A) edge[bend left=20] (C);
	\path [->] (F) edge[bend right=20] (H);
	\path [->] (K) edge[bend right=20] (M);
	\path [->] (P) edge[bend right=20] (R);
	\path [->] (A) edge[bend right=20] (S);
	\path [->] (B) edge[bend right=20] (T);
	\path [->] (C) edge[bend left=20] (U);
	\path [->] (S) edge[bend right=20] (U);
	\end{scope}
	\end{tikzpicture}
	\begin{center}
		Figure $1$: A $6$-star coloring of $C_3\square C_5$        
	\end{center}
\end{center}
}
\end{proof}

Given two integers $r$ and $s$, let $S(r,s)$ denote the set of all non-negative integer combinations of $r$ and $s$,    $ S(r,s) = \{\alpha r + \beta s \mid \alpha,\beta \in Z^{+}\cup \{0\}\}$.

The following result of Sylvester \cite{11}, has a key role in our proofs.

\begin{thm}{\rm \cite{11}} \label{prime}
{ If $r,s > 1$ are relatively prime integers, then $t\in S(r,s)$ for all $t \geq (s-1)(r-1)$. 
}\end{thm}

Now, we are in a position to prove the following result.

\begin{thm}\label{3n}
For every natural number $n\ge 4$  and $n \neq 5$, {$\chi_s(C_3\square C_n)=5$.
}\end{thm}

\begin{proof} 
{By Theorem \ref{prime}, every $n \ge 18$ is a linear combination of $4$ and $7$ with non-negative integer coefficients. Now, since 
the colors of the first three columns of Figure $2,\,(i)$ and $(iii)$ are the same, by considering some suitable copies of Figure $2,\,(i)$ and $(iii)$, one may obtain a $5$-star coloring of $C_3\square C_n$,  for $n\ge 18$.
Notice that every integer $n$, $n \in \{8,\ldots,17\} \backslash \{ 9 \}$ is a linear combination of $\{4,6,7\}$  and this implies $\chi_s(C_3\square C_n)\le 5$, ($n\neq 9$). Now, by considering  Figure $2,\,(iv)$ and using Theorem \ref{fertin}, the proof is complete.

\begin{center}
	\begin{tikzpicture}
	\hspace*{-8cm}
	\begin{scope}[every node/.style={circle,thick,draw,scale=0.5}]
	\node (A) at (0,0) {\textbf{1}};
	\node (B) at (0,1) {\textbf{3}};
	\node (C) at (0,2) {\textbf{4}};
	\node (F) at (1,0) {\textbf{5}};
	\node (G) at (1,1) {\textbf{1}};
	\node (H) at (1,2) {\textbf{2}};
	\node (K) at (2,0) {\textbf{4}};
	\node (L) at (2,1) {\textbf{2}};
	\node (M) at (2,2) {\textbf{3}};
	\node (P) at (3,0) {\textbf{2}};
	\node (Q) at (3,1) {\textbf{5}};
	\node (R) at (3,2) {\textbf{1}};
	\end{scope}
	\begin{scope}[>={[black]},
	every node/.style={fill=white,circle},
	every edge/.style={draw=black}]
	\path [->] (A) edge (B);
	\path [->] (B) edge (C);
	\path [->] (F) edge (G);
	\path [->] (G) edge (H);
	\path [->] (K) edge (L);
	\path [->] (L) edge (M);
	\path [->] (P) edge (Q);
	\path [->] (Q) edge (R);
	\path [->] (A) edge (F);
	\path [->] (F) edge (K);
	\path [->] (K) edge (P);
	\path [->] (B) edge (G);
	\path [->] (G) edge (L);
	\path [->] (L) edge (Q);
	\path [->] (C) edge (H);
	\path [->] (H) edge (M);
	\path [->] (M) edge (R);
	\path [->] (A) edge[bend left=20] (C);
	\path [->] (F) edge[bend right=20] (H);
	\path [->] (K) edge[bend right=20] (M);
	\path [->] (P) edge[bend right=20] (R);
	\path [->] (A) edge[bend right=20] (P);
	\path [->] (B) edge[bend right=20] (Q);
	\path [->] (C) edge[bend left=20] (R);
	\end{scope}
	\end{tikzpicture} 
 \hspace{2cm}
 \begin{tikzpicture}
	\begin{scope}[every node/.style={circle,thick,draw,scale=0.5}]
	\node (A) at (0,0) {\textbf{1}};
	\node (B) at (0,1) {\textbf{3}};
	\node (C) at (0,2) {\textbf{4}};
	\node (F) at (1,0) {\textbf{5}};
	\node (G) at (1,1) {\textbf{1}};
	\node (H) at (1,2) {\textbf{2}};
	\node (K) at (2,0) {\textbf{4}};
	\node (L) at (2,1) {\textbf{2}};
	\node (M) at (2,2) {\textbf{3}};
	\node (P) at (3,0) {\textbf{1}};
	\node (Q) at (3,1) {\textbf{5}};
	\node (R) at (3,2) {\textbf{4}};
	\node (S) at (4,0) {\textbf{3}};
	\node (T) at (4,1) {\textbf{1}};
	\node (U) at (4,2) {\textbf{2}};
	\node (V) at (5,0) {\textbf{4}};
	\node (W) at (5,1) {\textbf{2}};
	\node (X) at (5,2) {\textbf{5}};
	\end{scope}
	
	\begin{scope}[>={[black]},
	every node/.style={fill=white,circle},
	every edge/.style={draw=black}]
	\path [->] (A) edge (B);
	\path [->] (B) edge (C);
	\path [->] (F) edge (G);
	\path [->] (G) edge (H);
	\path [->] (K) edge (L);
	\path [->] (L) edge (M);
	\path [->] (P) edge (Q);
	\path [->] (Q) edge (R);
	\path [->] (A) edge (F);
	\path [->] (F) edge (K);
	\path [->] (K) edge (P);
	\path [->] (B) edge (G);
	\path [->] (G) edge (L);
	\path [->] (L) edge (Q);
	\path [->] (C) edge (H);
	\path [->] (H) edge (M);
	\path [->] (M) edge (R);
	\path [->] (S) edge (T);
	\path [->] (T) edge (U);
	\path [->] (V) edge (W);
	\path [->] (W) edge (X);
	\path [->] (R) edge (U);
	\path [->] (U) edge (X);
	\path [->] (Q) edge (T);
	\path [->] (T) edge (W);
	\path [->] (P) edge (S);
	\path [->] (S) edge (V);
	\path [->] (A) edge[bend left=20] (C);
	\path [->] (F) edge[bend right=20] (H);
	\path [->] (K) edge[bend right=20] (M);
	\path [->] (P) edge[bend right=20] (R);
	\path [->] (S) edge[bend right=20] (U);
	\path [->] (V) edge[bend right=20] (X);
	\path [->] (A) edge[bend right=20] (V);
	\path [->] (B) edge[bend right=20] (W);
	\path [->] (C) edge[bend left=20] (X);
	\end{scope}
	\end{tikzpicture}
\end{center}
\vspace{-.5cm}
\hspace*{3.5cm}$(i)$ \hspace*{6cm}$(ii)$

\begin{center}
	
	\begin{tikzpicture}
	\begin{scope}[every node/.style={circle,thick,draw,scale=0.5}]
	\node (A) at (0,0) {\textbf{1}};
	\node (B) at (0,1) {\textbf{3}};
	\node (C) at (0,2) {\textbf{4}};
	\node (D) at (1,0) {\textbf{5}};
	\node (E) at (1,1) {\textbf{1}};
	\node (F) at (1,2) {\textbf{2}};
	\node (G) at (2,0) {\textbf{4}};
	\node (H) at (2,1) {\textbf{2}};
	\node (I) at (2,2) {\textbf{3}};
	\node (J) at (3,0) {\textbf{2}};
	\node (K) at (3,1) {\textbf{5}};
	\node (L) at (3,2) {\textbf{1}};
	\node (M) at (4,0) {\textbf{3}};
	\node (N) at (4,1) {\textbf{4}};
	\node (O) at (4,2) {\textbf{5}};
	\node (P) at (5,0) {\textbf{4}};
	\node (Q) at (5,1) {\textbf{1}};
	\node (R) at (5,2) {\textbf{2}};
	\node (S) at (6,0) {\textbf{2}};
	\node (T) at (6,1) {\textbf{5}};
	\node (U) at (6,2) {\textbf{3}};
	\end{scope}
	
	\begin{scope}[>={[black]},
	every node/.style={fill=white,circle},
	every edge/.style={draw=black}]
	\path [->] (A) edge (B);
	\path [->] (B) edge (C);
	\path [->] (D) edge (E);
	\path [->] (E) edge (F);
	\path [->] (G) edge (H);
	\path [->] (H) edge (I);
	\path [->] (J) edge (K);
	\path [->] (K) edge (L);
	\path [->] (M) edge (N);
	\path [->] (N) edge (O);
	\path [->] (P) edge (Q);
	\path [->] (Q) edge (R);
	\path [->] (S) edge (T);
	\path [->] (T) edge (U);
	\path [->] (A) edge (D);
	\path [->] (D) edge (G);
	\path [->] (G) edge (J);
	\path [->] (J) edge (M);
	\path [->] (M) edge (P);
	\path [->] (P) edge (S);
	\path [->] (B) edge (E);
	\path [->] (E) edge (H);
	\path [->] (H) edge (K);
	\path [->] (K) edge (N);
	\path [->] (N) edge (Q);
	\path [->] (Q) edge (T);
	\path [->] (C) edge (F);
	\path [->] (F) edge (I);
	\path [->] (I) edge (L);
	\path [->] (L) edge (O);
	\path [->] (O) edge (R);
	\path [->] (R) edge (U);
	\path [->] (A) edge[bend left=20] (C);
	\path [->] (D) edge[bend right=20] (F);
	\path [->] (G) edge[bend right=20] (I);
	\path [->] (J) edge[bend right=20] (L);
	\path [->] (M) edge[bend right=20] (O);
	\path [->] (P) edge[bend right=20] (R);
	\path [->] (S) edge[bend right=20] (U);
	\path [->] (A) edge[bend right=20] (S);
	\path [->] (B) edge[bend right=20] (T);
	\path [->] (C) edge[bend left=20] (U);
	\end{scope}
	\end{tikzpicture}

$(iii)$
\vspace{1cm}

	\begin{tikzpicture}
	\begin{scope}[every node/.style={circle,thick,draw,scale=0.5}]
	\node (A) at (0,0) {\textbf{1}};
	\node (B) at (0,1) {\textbf{3}};
	\node (C) at (0,2) {\textbf{4}};
	\node (D) at (1,0) {\textbf{5}};
	\node (E) at (1,1) {\textbf{1}};
	\node (F) at (1,2) {\textbf{2}};
	\node (G) at (2,0) {\textbf{4}};
	\node (H) at (2,1) {\textbf{2}};
	\node (I) at (2,2) {\textbf{3}};
	\node (J) at (3,0) {\textbf{1}};
	\node (K) at (3,1) {\textbf{5}};
	\node (L) at (3,2) {\textbf{4}};
	\node (M) at (4,0) {\textbf{3}};
	\node (N) at (4,1) {\textbf{1}};
	\node (O) at (4,2) {\textbf{2}};
	\node (P) at (5,0) {\textbf{4}};
	\node (Q) at (5,1) {\textbf{2}};
	\node (R) at (5,2) {\textbf{5}};
	\node (S) at (6,0) {\textbf{5}};
	\node (T) at (6,1) {\textbf{3}};
	\node (U) at (6,2) {\textbf{1}};
	\node (V) at (7,0) {\textbf{3}};
	\node (W) at (7,1) {\textbf{4}};
	\node (X) at (7,2) {\textbf{2}};
	\node (Y) at (8,0) {\textbf{2}};
	\node (Z) at (8,1) {\textbf{1}};
	\node (ZZ) at (8,2) {\textbf{5}};
	\end{scope}
	
	\begin{scope}[>={[black]},
	every node/.style={fill=white,circle},
	every edge/.style={draw=black}]
	\path [->] (A) edge (B);
	\path [->] (B) edge (C);
	\path [->] (D) edge (E);
	\path [->] (E) edge (F);
	\path [->] (G) edge (H);
	\path [->] (H) edge (I);
	\path [->] (J) edge (K);
	\path [->] (K) edge (L);
	\path [->] (M) edge (N);
	\path [->] (N) edge (O);
	\path [->] (P) edge (Q);
	\path [->] (Q) edge (R);
	\path [->] (S) edge (T);
	\path [->] (T) edge (U);
	\path [->] (V) edge (W);
	\path [->] (W) edge (X);
	\path [->] (Y) edge (Z);
	\path [->] (Z) edge (ZZ);
	\path [->] (A) edge (D);
	\path [->] (D) edge (G);
	\path [->] (G) edge (J);
	\path [->] (J) edge (M);
	\path [->] (M) edge (P);
	\path [->] (P) edge (S);
	\path [->] (B) edge (E);
	\path [->] (E) edge (H);
	\path [->] (H) edge (K);
	\path [->] (K) edge (N);
	\path [->] (N) edge (Q);
	\path [->] (Q) edge (T);
	\path [->] (C) edge (F);
	\path [->] (F) edge (I);
	\path [->] (I) edge (L);
	\path [->] (L) edge (O);
	\path [->] (O) edge (R);
	\path [->] (R) edge (U);
	\path [->] (S) edge (V);
	\path [->] (V) edge (Y);
	\path [->] (T) edge (W);
	\path [->] (W) edge (Z);
	\path [->] (U) edge (X);
	\path [->] (X) edge (ZZ);
	\path [->] (A) edge[bend left=20] (C);
	\path [->] (D) edge[bend right=20] (F);
	\path [->] (G) edge[bend right=20] (I);
	\path [->] (J) edge[bend right=20] (L);
	\path [->] (M) edge[bend right=20] (O);
	\path [->] (P) edge[bend right=20] (R);
	\path [->] (S) edge[bend right=20] (U);
	\path [->] (V) edge[bend right=20] (X);
	\path [->] (Y) edge[bend right=20] (ZZ);
	\path [->] (A) edge[bend right=10] (Y);
	\path [->] (B) edge[bend right=10] (Z);
	\path [->] (C) edge[bend left=10] (ZZ);
	\end{scope}
	\end{tikzpicture}

$(iv)$
\end{center}

\begin{center}
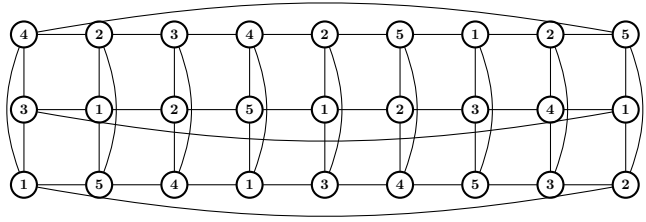
Figure $2$:  $5$-star colorings of $C_3\square C_n$, for $n=4,6,7,9$
\end{center}

}
\end{proof}

\begin{thm}\label{4n}
{For every natural number $n\ge 3$ , $\chi_s(C_4\square C_n)=5$.
}\end{thm}

\begin{proof} 
{By Theorem \ref{3n}, we can assume that $n\ge 4$. Then, by Theorem \ref{prime}, every $n \ge 12$ is a linear combination of $4$ and $5$ with non-negative integer coefficients. Now, since 
the colors of the first three columns of Figure $3,\,(i)$ and $(ii)$ are the same, by considering some suitable copies of Figure $3,\, (i)$ and $(ii)$, one may obtain a $5$-star coloring of $C_4\square C_n$, for $n\ge 12$.
Notice that every integer $n=4,5,8,9,10 $ is a linear combination of $\{4,5\}$  and for  $n=6,7,11$, by considering Figure $3,\,(iii)$, $(iv)$ and $(v)$ and using Theorem \ref{fertin}, the proof is complete.
\\
\begin{center}
\begin{tikzpicture}
\begin{scope}[every node/.style={circle,thick,draw,scale=0.5}]
\node (A) at (0,0) {\textbf{4}};
\node (B) at (0,1) {\textbf{2}};
\node (C) at (0,2) {\textbf{4}};
\node (D) at (0,3) {\textbf{1}};
\node (F) at (1,0) {\textbf{5}};
\node (G) at (1,1) {\textbf{3}};
\node (H) at (1,2) {\textbf{5}};
\node (I) at (1,3) {\textbf{2}};
\node (K) at (2,0) {\textbf{1}};
\node (L) at (2,1) {\textbf{4}};
\node (M) at (2,2) {\textbf{1}};
\node (N) at (2,3) {\textbf{3}};
\node (P) at (3,0) {\textbf{5}};
\node (Q) at (3,1) {\textbf{3}};
\node (R) at (3,2) {\textbf{5}};
\node (S) at (3,3) {\textbf{2}};
\end{scope}

\begin{scope}[>={[black]},
every node/.style={fill=white,circle},
every edge/.style={draw=black}]
\path [->] (A) edge (B);
\path [->] (B) edge (C);
\path [->] (C) edge (D);
\path [->] (F) edge (G);
\path [->] (G) edge (H);
\path [->] (H) edge (I);;
\path [->] (K) edge (L);
\path [->] (L) edge (M);
\path [->] (M) edge (N);
\path [->] (P) edge (Q);
\path [->] (Q) edge (R);
\path [->] (R) edge (S);
\path [->] (A) edge (F);
\path [->] (F) edge (K);
\path [->] (K) edge (P);
\path [->] (B) edge (G);
\path [->] (G) edge (L);
\path [->] (L) edge (Q);
\path [->] (C) edge (H);
\path [->] (H) edge (M);
\path [->] (M) edge (R);
\path [->] (D) edge (I);
\path [->] (I) edge (N);
\path [->] (N) edge (S);
\path [->] (A) edge[bend left=20] (D);
\path [->] (F) edge[bend right=20] (I);
\path [->] (K) edge[bend right=20] (N);
\path [->] (P) edge[bend right=20] (S);
\path [->] (A) edge[bend right=20] (P);
\path [->] (B) edge[bend right=20] (Q);
\path [->] (C) edge[bend right=20] (R);
\path [->] (D) edge[bend left=20] (S);
\end{scope}
\end{tikzpicture}
\hspace{3.5cm}
\begin{tikzpicture}
\begin{scope}[every node/.style={circle,thick,draw,scale=0.5}]
\node (A) at (0,0) {\textbf{4}};
\node (B) at (0,1) {\textbf{2}};
\node (C) at (0,2) {\textbf{4}};
\node (D) at (0,3) {\textbf{1}};
\node (F) at (1,0) {\textbf{5}};
\node (G) at (1,1) {\textbf{3}};
\node (H) at (1,2) {\textbf{5}};
\node (I) at (1,3) {\textbf{2}};
\node (K) at (2,0) {\textbf{1}};
\node (L) at (2,1) {\textbf{4}};
\node (M) at (2,2) {\textbf{1}};
\node (N) at (2,3) {\textbf{3}};
\node (P) at (3,0) {\textbf{2}};
\node (Q) at (3,1) {\textbf{5}};
\node (R) at (3,2) {\textbf{2}};
\node (S) at (3,3) {\textbf{4}};
\node (T) at (4,0) {\textbf{3}};
\node (U) at (4,1) {\textbf{1}};
\node (V) at (4,2) {\textbf{3}};
\node (W) at (4,3) {\textbf{5}};
\end{scope}

\begin{scope}[>={[black]},
every node/.style={fill=white,circle},
every edge/.style={draw=black}]
\path [->] (A) edge (B);
\path [->] (B) edge (C);
\path [->] (C) edge (D);
\path [->] (F) edge (G);
\path [->] (G) edge (H);
\path [->] (H) edge (I);;
\path [->] (K) edge (L);
\path [->] (L) edge (M);
\path [->] (M) edge (N);
\path [->] (P) edge (Q);
\path [->] (Q) edge (R);
\path [->] (R) edge (S);
\path [->] (T) edge (U);
\path [->] (U) edge (V);
\path [->] (V) edge (W);
\path [->] (A) edge (F);
\path [->] (F) edge (K);
\path [->] (K) edge (P);
\path [->] (B) edge (G);
\path [->] (G) edge (L);
\path [->] (L) edge (Q);
\path [->] (C) edge (H);
\path [->] (H) edge (M);
\path [->] (M) edge (R);
\path [->] (D) edge (I);
\path [->] (I) edge (N);
\path [->] (N) edge (S);
\path [->] (P) edge (T);
\path [->] (Q) edge (U);
\path [->] (R) edge (V);
\path [->] (S) edge (W);
\path [->] (A) edge[bend left=20] (D);
\path [->] (F) edge[bend right=20] (I);
\path [->] (K) edge[bend right=20] (N);
\path [->] (P) edge[bend right=20] (S);
\path [->] (T) edge[bend right=20] (W);
\path [->] (A) edge[bend right=20] (T);
\path [->] (B) edge[bend right=20] (U);
\path [->] (C) edge[bend right=20] (V);
\path [->] (D) edge[bend left=20] (W);
\end{scope}
\end{tikzpicture}
\end{center}
\vspace{-.5cm}
\hspace*{3cm}$(i)$ \hspace*{7.5cm}$(ii)$

\begin{center}

\begin{tikzpicture}
\hspace{-2cm}
\begin{scope}[every node/.style={circle,thick,draw,scale=0.5}]
\node (A) at (0,0) {\textbf{4}};
\node (B) at (0,1) {\textbf{5}};
\node (C) at (0,2) {\textbf{2}};
\node (D) at (0,3) {\textbf{1}};
\node (F) at (1,0) {\textbf{2}};
\node (G) at (1,1) {\textbf{1}};
\node (H) at (1,2) {\textbf{4}};
\node (I) at (1,3) {\textbf{3}};
\node (K) at (2,0) {\textbf{5}};
\node (L) at (2,1) {\textbf{3}};
\node (M) at (2,2) {\textbf{2}};
\node (N) at (2,3) {\textbf{1}};
\node (P) at (3,0) {\textbf{2}};
\node (Q) at (3,1) {\textbf{1}};
\node (R) at (3,2) {\textbf{5}};
\node (S) at (3,3) {\textbf{4}};
\node (T) at (4,0) {\textbf{3}};
\node (U) at (4,1) {\textbf{4}};
\node (V) at (4,2) {\textbf{2}};
\node (W) at (4,3) {\textbf{1}};
\node (A1) at (5,0) {\textbf{2}};
\node (B1) at (5,1) {\textbf{1}};
\node (C1) at (5,2) {\textbf{3}};
\node (D1) at (5,3) {\textbf{5}};
\end{scope}

\begin{scope}[>={[black]},
every node/.style={fill=white,circle},
every edge/.style={draw=black}]
\path [->] (A) edge (B);
\path [->] (B) edge (C);
\path [->] (C) edge (D);
\path [->] (F) edge (G);
\path [->] (G) edge (H);
\path [->] (H) edge (I);;
\path [->] (K) edge (L);
\path [->] (L) edge (M);
\path [->] (M) edge (N);
\path [->] (P) edge (Q);
\path [->] (Q) edge (R);
\path [->] (R) edge (S);
\path [->] (T) edge (U);
\path [->] (U) edge (V);
\path [->] (V) edge (W);
\path [->] (A1) edge (B1);
\path [->] (B1) edge (C1);
\path [->] (C1) edge (D1);
\path [->] (A) edge (F);
\path [->] (F) edge (K);
\path [->] (K) edge (P);
\path [->] (T) edge (A1);
\path [->] (B) edge (G);
\path [->] (G) edge (L);
\path [->] (L) edge (Q);
\path [->] (U) edge (B1);
\path [->] (C) edge (H);
\path [->] (H) edge (M);
\path [->] (M) edge (R);
\path [->] (V) edge (C1);
\path [->] (D) edge (I);
\path [->] (I) edge (N);
\path [->] (N) edge (S);
\path [->] (W) edge (D1);
\path [->] (P) edge (T);
\path [->] (Q) edge (U);
\path [->] (R) edge (V);
\path [->] (S) edge (W);
\path [->] (A) edge[bend left=20] (D);
\path [->] (F) edge[bend right=20] (I);
\path [->] (K) edge[bend right=20] (N);
\path [->] (P) edge[bend right=20] (S);
\path [->] (T) edge[bend right=20] (W);
\path [->] (A1) edge[bend right=20] (D1);
\path [->] (A) edge[bend right=20] (A1);
\path [->] (B) edge[bend right=20] (B1);
\path [->] (C) edge[bend right=20] (C1);
\path [->] (D) edge[bend left=20] (D1);
\end{scope}
\end{tikzpicture} 
 \hspace{2cm}
 \begin{tikzpicture}
\begin{scope}[every node/.style={circle,thick,draw,scale=0.5}]
\node (A) at (0,0) {\textbf{5}};
\node (B) at (0,1) {\textbf{1}};
\node (C) at (0,2) {\textbf{4}};
\node (D) at (0,3) {\textbf{1}};
\node (F) at (1,0) {\textbf{1}};
\node (G) at (1,1) {\textbf{3}};
\node (H) at (1,2) {\textbf{1}};
\node (I) at (1,3) {\textbf{2}};
\node (K) at (2,0) {\textbf{4}};
\node (L) at (2,1) {\textbf{1}};
\node (M) at (2,2) {\textbf{5}};
\node (N) at (2,3) {\textbf{1}};
\node (P) at (3,0) {\textbf{1}};
\node (Q) at (3,1) {\textbf{2}};
\node (R) at (3,2) {\textbf{1}};
\node (S) at (3,3) {\textbf{3}};
\node (T) at (4,0) {\textbf{5}};
\node (U) at (4,1) {\textbf{1}};
\node (V) at (4,2) {\textbf{4}};
\node (W) at (4,3) {\textbf{1}};
\node (A1) at (5,0) {\textbf{1}};
\node (B1) at (5,1) {\textbf{3}};
\node (C1) at (5,2) {\textbf{1}};
\node (D1) at (5,3) {\textbf{2}};
\node (F1) at (6,0) {\textbf{4}};
\node (G1) at (6,1) {\textbf{2}};
\node (H1) at (6,2) {\textbf{5}};
\node (I1) at (6,3) {\textbf{3}};
\end{scope}

\begin{scope}[>={[black]},
every node/.style={fill=white,circle},
every edge/.style={draw=black}]
\path [->] (A) edge (B);
\path [->] (B) edge (C);
\path [->] (C) edge (D);
\path [->] (F) edge (G);
\path [->] (G) edge (H);
\path [->] (H) edge (I);;
\path [->] (K) edge (L);
\path [->] (L) edge (M);
\path [->] (M) edge (N);
\path [->] (P) edge (Q);
\path [->] (Q) edge (R);
\path [->] (R) edge (S);
\path [->] (T) edge (U);
\path [->] (U) edge (V);
\path [->] (V) edge (W);
\path [->] (A1) edge (B1);
\path [->] (B1) edge (C1);
\path [->] (C1) edge (D1);
\path [->] (F1) edge (G1);
\path [->] (G1) edge (H1);
\path [->] (H1) edge (I1);
\path [->] (A) edge (F);
\path [->] (F) edge (K);
\path [->] (K) edge (P);
\path [->] (T) edge (A1);
\path [->] (A1) edge (F1);
\path [->] (B) edge (G);
\path [->] (G) edge (L);
\path [->] (L) edge (Q);
\path [->] (U) edge (B1);
\path [->] (B1) edge (G1);
\path [->] (C) edge (H);
\path [->] (H) edge (M);
\path [->] (M) edge (R);
\path [->] (V) edge (C1);
\path [->] (C1) edge (H1);
\path [->] (D) edge (I);
\path [->] (I) edge (N);
\path [->] (N) edge (S);
\path [->] (W) edge (D1);
\path [->] (D1) edge (I1);
\path [->] (P) edge (T);
\path [->] (Q) edge (U);
\path [->] (R) edge (V);
\path [->] (S) edge (W);
\path [->] (A) edge[bend left=20] (D);
\path [->] (F) edge[bend right=20] (I);
\path [->] (K) edge[bend right=20] (N);
\path [->] (P) edge[bend right=20] (S);
\path [->] (T) edge[bend right=20] (W);
\path [->] (A1) edge[bend right=20] (D1);
\path [->] (F1) edge[bend right=20] (I1);
\path [->] (A) edge[bend right=20] (F1);
\path [->] (B) edge[bend right=20] (G1);
\path [->] (C) edge[bend right=20] (H1);
\path [->] (D) edge[bend left=20] (I1);
\end{scope}
\end{tikzpicture}
\end{center}
\vspace{-.5cm}
\hspace*{2.7cm}$(iii)$ \hspace*{7.8cm}$(iv)$

\begin{center}
	\begin{tikzpicture}
	\begin{scope}[every node/.style={circle,thick,draw,scale=0.5}]
	\node (A) at (0,0) {\textbf{4}};
	\node (B) at (0,1) {\textbf{2}};
	\node (C) at (0,2) {\textbf{4}};
	\node (D) at (0,3) {\textbf{1}};
	\node (F) at (1,0) {\textbf{5}};
	\node (G) at (1,1) {\textbf{3}};
	\node (H) at (1,2) {\textbf{5}};
	\node (I) at (1,3) {\textbf{2}};
	\node (K) at (2,0) {\textbf{1}};
	\node (L) at (2,1) {\textbf{4}};
	\node (M) at (2,2) {\textbf{1}};
	\node (N) at (2,3) {\textbf{3}};
	\node (P) at (3,0) {\textbf{2}};
	\node (Q) at (3,1) {\textbf{5}};
	\node (R) at (3,2) {\textbf{2}};
	\node (S) at (3,3) {\textbf{4}};
	\node (T) at (4,0) {\textbf{3}};
	\node (U) at (4,1) {\textbf{1}};
	\node (V) at (4,2) {\textbf{3}};
	\node (W) at (4,3) {\textbf{5}};
	\node (A1) at (5,0) {\textbf{4}};
	\node (B1) at (5,1) {\textbf{2}};
	\node (C1) at (5,2) {\textbf{4}};
	\node (D1) at (5,3) {\textbf{1}};
	\node (F1) at (6,0) {\textbf{5}};
	\node (G1) at (6,1) {\textbf{3}};
	\node (H1) at (6,2) {\textbf{5}};
	\node (I1) at (6,3) {\textbf{2}};
	\node (K1) at (7,0) {\textbf{1}};
	\node (L1) at (7,1) {\textbf{4}};
	\node (M1) at (7,2) {\textbf{1}};
	\node (N1) at (7,3) {\textbf{3}};
	\node (P1) at (8,0) {\textbf{2}};
	\node (Q1) at (8,1) {\textbf{5}};
	\node (R1) at (8,2) {\textbf{2}};
	\node (S1) at (8,3) {\textbf{4}};
	\node (T1) at (9,0) {\textbf{1}};
	\node (U1) at (9,1) {\textbf{4}};
	\node (V1) at (9,2) {\textbf{1}};
	\node (W1) at (9,3) {\textbf{3}};
	\node (A2) at (10,0) {\textbf{5}};
	\node (B2) at (10,1) {\textbf{3}};
	\node (C2) at (10,2) {\textbf{5}};
	\node (D2) at (10,3) {\textbf{2}};
	\end{scope}
	
	\begin{scope}[>={[black]},
	every node/.style={fill=white,circle},
	every edge/.style={draw=black}]
	\path [->] (A) edge (B);
	\path [->] (B) edge (C);
	\path [->] (C) edge (D);
	\path [->] (F) edge (G);
	\path [->] (G) edge (H);
	\path [->] (H) edge (I);;
	\path [->] (K) edge (L);
	\path [->] (L) edge (M);
	\path [->] (M) edge (N);
	\path [->] (P) edge (Q);
	\path [->] (Q) edge (R);
	\path [->] (R) edge (S);
	\path [->] (T) edge (U);
	\path [->] (U) edge (V);
	\path [->] (V) edge (W);
	\path [->] (A1) edge (B1);
	\path [->] (B1) edge (C1);
	\path [->] (C1) edge (D1);
	\path [->] (F1) edge (G1);
	\path [->] (G1) edge (H1);
	\path [->] (H1) edge (I1);
	\path [->] (K1) edge (L1);
	\path [->] (L1) edge (M1);
	\path [->] (M1) edge (N1);
	\path [->] (P1) edge (Q1);
	\path [->] (Q1) edge (R1);
	\path [->] (R1) edge (S1);
	\path [->] (T1) edge (U1);
	\path [->] (U1) edge (V1);
	\path [->] (V1) edge (W1);
	\path [->] (A2) edge (B2);
	\path [->] (B2) edge (C2);
	\path [->] (C2) edge (D2);
	\path [->] (A) edge (F);
	\path [->] (F) edge (K);
	\path [->] (K) edge (P);
	\path [->] (T) edge (A1);
	\path [->] (A1) edge (F1);
	\path [->] (B) edge (G);
	\path [->] (G) edge (L);
	\path [->] (L) edge (Q);
	\path [->] (U) edge (B1);
	\path [->] (B1) edge (G1);
	\path [->] (C) edge (H);
	\path [->] (H) edge (M);
	\path [->] (M) edge (R);
	\path [->] (V) edge (C1);
	\path [->] (C1) edge (H1);
	\path [->] (D) edge (I);
	\path [->] (I) edge (N);
	\path [->] (N) edge (S);
	\path [->] (W) edge (D1);
	\path [->] (D1) edge (I1);
	\path [->] (P) edge (T);
	\path [->] (Q) edge (U);
	\path [->] (R) edge (V);
	\path [->] (S) edge (W);
	\path [->] (F1) edge (K1);
	\path [->] (K1) edge (P1);
	\path [->] (P1) edge (T1);
	\path [->] (T1) edge (A2);
	\path [->] (G1) edge (L1);
	\path [->] (L1) edge (Q1);
	\path [->] (Q1) edge (U1);
	\path [->] (U1) edge (B2);
	\path [->] (H1) edge (M1);
	\path [->] (M1) edge (R1);
	\path [->] (R1) edge (V1);
	\path [->] (V1) edge (C2);
	\path [->] (I1) edge (N1);
	\path [->] (N1) edge (S1);
	\path [->] (S1) edge (W1);
	\path [->] (W1) edge (D2);
	\path [->] (A) edge[bend left=20] (D);
	\path [->] (F) edge[bend right=20] (I);
	\path [->] (K) edge[bend right=20] (N);
	\path [->] (P) edge[bend right=20] (S);
	\path [->] (T) edge[bend right=20] (W);
	\path [->] (A1) edge[bend right=20] (D1);
	\path [->] (F1) edge[bend right=20] (I1);
	\path [->] (K1) edge[bend right=20] (N1);
	\path [->] (P1) edge[bend right=20] (S1);
	\path [->] (T1) edge[bend right=20] (W1);
	\path [->] (A2) edge[bend right=20] (D2);
	\path [->] (A) edge[bend right=10] (A2);
	\path [->] (B) edge[bend right=10] (B2);
	\path [->] (C) edge[bend right=10] (C2);
	\path [->] (D) edge[bend left=10] (D2);
	\end{scope}
	\end{tikzpicture}

$(v)$
\end{center}

\begin{center}
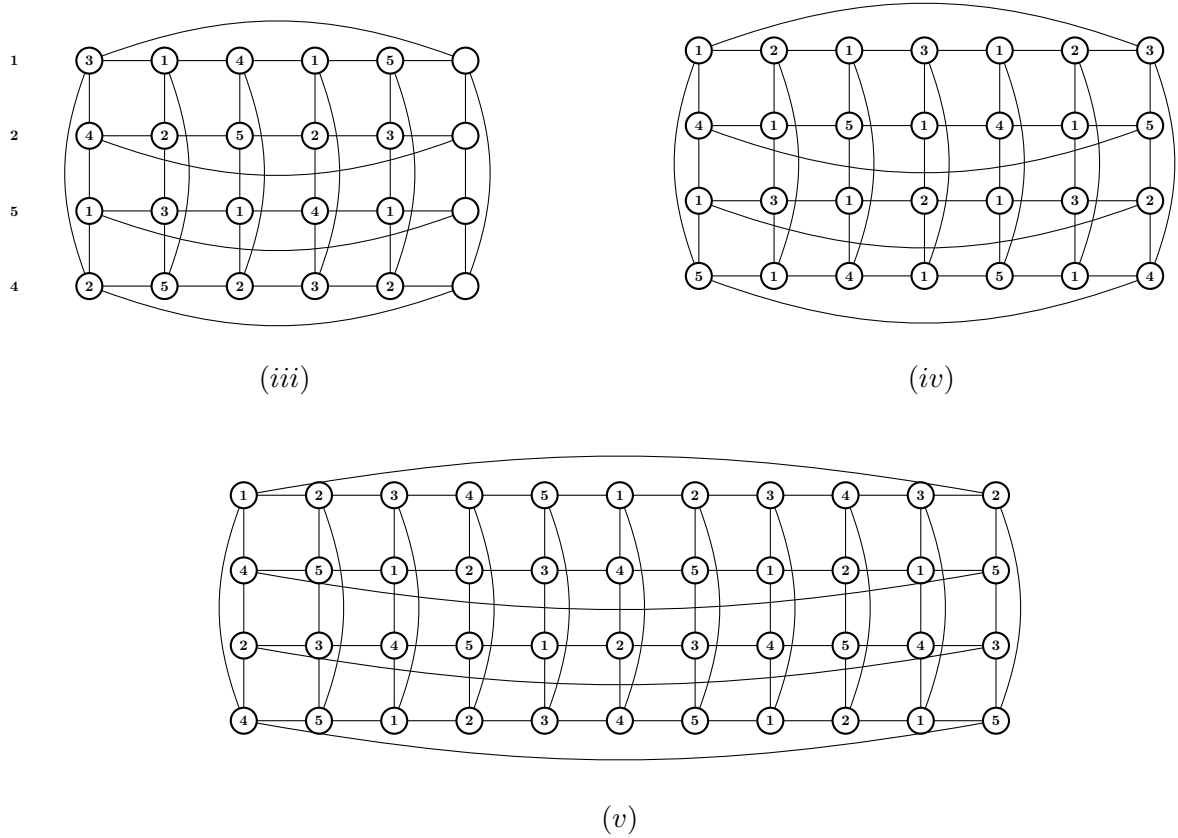
Figure $3$:  $5$-star colorings of $C_4\square C_n$, for $n=4,5,6,7,11$
\end{center}

}  
\end{proof}


\begin{thm}\label{5n}
{For every natural number $n\ge 4$, $\chi_s(C_5\square C_n)=5$.
}\end{thm}

\begin{proof} 
{By Theorem \ref{4n}, we can assume that $n\ge 5$.  Theorem \ref{prime} states that every $n \ge 12$ is a linear combination of $4$ and $5$ with non-negative integer coefficients. Now, since 
the colors of the first three columns of Figure $4 ,\, (i)$ and $(ii)$ are the same, by considering some suitable copies of Figure $4 ,\, (i)$ and $(ii)$, one may obtain a $5$-star coloring of $C_5\square C_n$, for $n\ge 12$.
Notice that every integer $n=4,5,8,9,10 $, is a linear combination of $\{4,5\}$ and also for   $n=6,7,11 $, by considering Figure $4,\,(iii)$, $(iv)$ and $(v)$ and using Theorem \ref{fertin}, the proof is complete.

\begin{center}
\begin{tikzpicture}

\begin{scope}[every node/.style={circle,thick,draw,scale=0.5}]
\node (A) at (0,0) {\textbf{3}};
\node (B) at (0,1) {\textbf{5}};
\node (C) at (0,2) {\textbf{2}};
\node (D) at (0,3) {\textbf{4}};
\node (E) at (0,4) {\textbf{1}};
\node (F) at (1,0) {\textbf{4}};
\node (G) at (1,1) {\textbf{1}};
\node (H) at (1,2) {\textbf{3}};
\node (I) at (1,3) {\textbf{5}};
\node (J) at (1,4) {\textbf{2}};
\node (K) at (2,0) {\textbf{5}};
\node (L) at (2,1) {\textbf{2}};
\node (M) at (2,2) {\textbf{4}};
\node (N) at (2,3) {\textbf{1}};
\node (O) at (2,4) {\textbf{3}};
\node (P) at (3,0) {\textbf{4}};
\node (Q) at (3,1) {\textbf{1}};
\node (R) at (3,2) {\textbf{3}};
\node (S) at (3,3) {\textbf{5}};
\node (T) at (3,4) {\textbf{2}};
\end{scope}

\begin{scope}[>={[black]},
every node/.style={fill=white,circle},
every edge/.style={draw=black}]
\path [->] (A) edge (B);
\path [->] (B) edge (C);
\path [->] (C) edge (D);
\path [->] (D) edge (E);
\path [->] (F) edge (G);
\path [->] (G) edge (H);
\path [->] (H) edge (I);
\path [->] (I) edge (J);
\path [->] (K) edge (L);
\path [->] (L) edge (M);
\path [->] (M) edge (N);
\path [->] (N) edge (O);
\path [->] (P) edge (Q);
\path [->] (Q) edge (R);
\path [->] (R) edge (S);
\path [->] (S) edge (T);
\path [->] (A) edge (F);
\path [->] (F) edge (K);
\path [->] (K) edge (P);
\path [->] (B) edge (G);
\path [->] (G) edge (L);
\path [->] (L) edge (Q);
\path [->] (C) edge (H);
\path [->] (H) edge (M);
\path [->] (M) edge (R);
\path [->] (D) edge (I);
\path [->] (I) edge (N);
\path [->] (N) edge (S);
\path [->] (E) edge (J);
\path [->] (J) edge (O);
\path [->] (O) edge (T);
\path [->] (A) edge[bend left=20] (E);
\path [->] (F) edge[bend right=20] (J);
\path [->] (K) edge[bend right=20] (O);
\path [->] (P) edge[bend right=20] (T);
\path [->] (A) edge[bend right=20] (P);
\path [->] (B) edge[bend right=20] (Q);
\path [->] (C) edge[bend right=20] (R);
\path [->] (D) edge[bend right=20] (S);
\path [->] (E) edge[bend left=20] (T);
\end{scope}

\end{tikzpicture}
\hspace{3.5cm}
\begin{tikzpicture}

\begin{scope}[every node/.style={circle,thick,draw,scale=0.5}]
\node (A) at (0,0) {\textbf{3}};
\node (B) at (0,1) {\textbf{5}};
\node (C) at (0,2) {\textbf{2}};
\node (D) at (0,3) {\textbf{4}};
\node (E) at (0,4) {\textbf{1}};
\node (F) at (1,0) {\textbf{4}};
\node (G) at (1,1) {\textbf{1}};
\node (H) at (1,2) {\textbf{3}};
\node (I) at (1,3) {\textbf{5}};
\node (J) at (1,4) {\textbf{2}};
\node (K) at (2,0) {\textbf{5}};
\node (L) at (2,1) {\textbf{2}};
\node (M) at (2,2) {\textbf{4}};
\node (N) at (2,3) {\textbf{1}};
\node (O) at (2,4) {\textbf{3}};
\node (P) at (3,0) {\textbf{1}};
\node (Q) at (3,1) {\textbf{3}};
\node (R) at (3,2) {\textbf{5}};
\node (S) at (3,3) {\textbf{2}};
\node (T) at (3,4) {\textbf{4}};
\node (U) at (4,0) {\textbf{2}};
\node (V) at (4,1) {\textbf{4}};
\node (W) at (4,2) {\textbf{1}};
\node (X) at (4,3) {\textbf{3}};
\node (Y) at (4,4) {\textbf{5}};
\end{scope}

\begin{scope}[>={[black]},
every node/.style={fill=white,circle},
every edge/.style={draw=black}]
\path [->] (A) edge (B);
\path [->] (B) edge (C);
\path [->] (C) edge (D);
\path [->] (D) edge (E);
\path [->] (F) edge (G);
\path [->] (G) edge (H);
\path [->] (H) edge (I);
\path [->] (I) edge (J);
\path [->] (K) edge (L);
\path [->] (L) edge (M);
\path [->] (M) edge (N);
\path [->] (N) edge (O);
\path [->] (P) edge (Q);
\path [->] (Q) edge (R);
\path [->] (R) edge (S);
\path [->] (S) edge (T);
\path [->] (U) edge (V);
\path [->] (V) edge (W);
\path [->] (W) edge (X);
\path [->] (X) edge (Y);
\path [->] (A) edge (F);
\path [->] (F) edge (K);
\path [->] (K) edge (P);
\path [->] (P) edge (U);
\path [->] (B) edge (G);
\path [->] (G) edge (L);
\path [->] (L) edge (Q);
\path [->] (Q) edge (V);
\path [->] (C) edge (H);
\path [->] (H) edge (M);
\path [->] (M) edge (R);
\path [->] (R) edge (W);
\path [->] (D) edge (I);
\path [->] (I) edge (N);
\path [->] (N) edge (S);
\path [->] (S) edge (X);
\path [->] (E) edge (J);
\path [->] (J) edge (O);
\path [->] (O) edge (T);
\path [->] (T) edge (Y);
\path [->] (A) edge[bend left=20] (E);
\path [->] (F) edge[bend right=20] (J);
\path [->] (K) edge[bend right=20] (O);
\path [->] (P) edge[bend right=20] (T);
\path [->] (U) edge[bend right=20] (Y);
\path [->] (A) edge[bend right=20] (U);
\path [->] (B) edge[bend right=20] (V);
\path [->] (C) edge[bend right=20] (W);
\path [->] (D) edge[bend right=20] (X);
\path [->] (E) edge[bend left=20] (Y);
\end{scope}
\end{tikzpicture}
\end{center}
\vspace{-.5cm}
\hspace*{2.9cm}$(i)$ \hspace*{7.7cm}$(ii)$

\begin{center}
	\begin{tikzpicture}
	\begin{scope}[every node/.style={circle,thick,draw,scale=0.5}]
	\node (A) at (0,0) {\textbf{2}};
	\node (B) at (0,1) {\textbf{3}};
	\node (C) at (0,2) {\textbf{5}};
	\node (D) at (0,3) {\textbf{2}};
	\node (E) at (0,4) {\textbf{1}};
	\node (F) at (1,0) {\textbf{5}};
	\node (G) at (1,1) {\textbf{2}};
	\node (H) at (1,2) {\textbf{1}};
	\node (I) at (1,3) {\textbf{4}};
	\node (J) at (1,4) {\textbf{3}};
	\node (K) at (2,0) {\textbf{2}};
	\node (L) at (2,1) {\textbf{4}};
	\node (M) at (2,2) {\textbf{3}};
	\node (N) at (2,3) {\textbf{2}};
	\node (O) at (2,4) {\textbf{1}};
	\node (P) at (3,0) {\textbf{3}};
	\node (Q) at (3,1) {\textbf{2}};
	\node (R) at (3,2) {\textbf{1}};
	\node (S) at (3,3) {\textbf{5}};
	\node (S1) at (3,4) {\textbf{4}};
	\node (T) at (4,0) {\textbf{2}};
	\node (U) at (4,1) {\textbf{5}};
	\node (V) at (4,2) {\textbf{4}};
	\node (W) at (4,3) {\textbf{2}};
	\node (X) at (4,4) {\textbf{1}};
	\node (A1) at (5,0) {\textbf{4}};
	\node (B1) at (5,1) {\textbf{2}};
	\node (C1) at (5,2) {\textbf{1}};
	\node (D1) at (5,3) {\textbf{3}};
	\node (E1) at (5,4) {\textbf{5}};
	\end{scope}
	
	\begin{scope}[>={[black]},
	every node/.style={fill=white,circle},
	every edge/.style={draw=black}]
	\path [->] (A) edge (B);
	\path [->] (B) edge (C);
	\path [->] (C) edge (D);
	\path [->] (D) edge (E);
	\path [->] (F) edge (G);
	\path [->] (G) edge (H);
	\path [->] (H) edge (I);
	\path [->] (I) edge (J);
	\path [->] (K) edge (L);
	\path [->] (L) edge (M);
	\path [->] (M) edge (N);
	\path [->] (N) edge (O);
	\path [->] (P) edge (Q);
	\path [->] (Q) edge (R);
	\path [->] (R) edge (S);
	\path [->] (S) edge (S1);
	\path [->] (T) edge (U);
	\path [->] (U) edge (V);
	\path [->] (V) edge (W);
	\path [->] (W) edge (X);
	\path [->] (A1) edge (B1);
	\path [->] (B1) edge (C1);
	\path [->] (C1) edge (D1);
	\path [->] (D1) edge (E1);
	\path [->] (A) edge (F);
	\path [->] (F) edge (K);
	\path [->] (K) edge (P);
	\path [->] (T) edge (A1);
	\path [->] (B) edge (G);
	\path [->] (G) edge (L);
	\path [->] (L) edge (Q);
	\path [->] (U) edge (B1);
	\path [->] (C) edge (H);
	\path [->] (H) edge (M);
	\path [->] (M) edge (R);
	\path [->] (V) edge (C1);
	\path [->] (D) edge (I);
	\path [->] (I) edge (N);
	\path [->] (N) edge (S);
	\path [->] (W) edge (D1);
	\path [->] (P) edge (T);
	\path [->] (Q) edge (U);
	\path [->] (R) edge (V);
	\path [->] (S) edge (W);
	\path [->] (E) edge (J);
	\path [->] (J) edge (O);
	\path [->] (O) edge (S1);
	\path [->] (S1) edge (X);
	\path [->] (X) edge (E1);
	\path [->] (A) edge[bend left=20] (E);
	\path [->] (F) edge[bend right=20] (J);
	\path [->] (K) edge[bend right=20] (O);
	\path [->] (P) edge[bend right=20] (S1);
	\path [->] (T) edge[bend right=20] (X);
	\path [->] (A1) edge[bend right=20] (E1);
	\path [->] (A) edge[bend right=20] (A1);
	\path [->] (B) edge[bend right=20] (B1);
	\path [->] (C) edge[bend right=20] (C1);
	\path [->] (D) edge[bend right=20] (D1);
	\path [->] (E) edge[bend left=20] (E1);
	\end{scope}
	\end{tikzpicture}
	\hspace{.5cm}
	\begin{tikzpicture}
	\begin{scope}[every node/.style={circle,thick,draw,scale=0.5}]
	\node (A) at (0,0) {\textbf{2}};
	\node (B) at (0,1) {\textbf{1}};
	\node (C) at (0,2) {\textbf{4}};
	\node (D) at (0,3) {\textbf{5}};
	\node (E) at (0,4) {\textbf{1}};
	\node (F) at (1,0) {\textbf{5}};
	\node (G) at (1,1) {\textbf{4}};
	\node (H) at (1,2) {\textbf{2}};
	\node (I) at (1,3) {\textbf{1}};
	\node (J) at (1,4) {\textbf{3}};
	\node (K) at (2,0) {\textbf{1}};
	\node (L) at (2,1) {\textbf{5}};
	\node (M) at (2,2) {\textbf{3}};
	\node (N) at (2,3) {\textbf{2}};
	\node (O) at (2,4) {\textbf{4}};
	\node (P) at (3,0) {\textbf{5}};
	\node (Q) at (3,1) {\textbf{2}};
	\node (R) at (3,2) {\textbf{1}};
	\node (S) at (3,3) {\textbf{5}};
	\node (S1) at (3,4) {\textbf{3}};
	\node (T) at (4,0) {\textbf{4}};
	\node (U) at (4,1) {\textbf{3}};
	\node (V) at (4,2) {\textbf{5}};
	\node (W) at (4,3) {\textbf{4}};
	\node (X) at (4,4) {\textbf{2}};
	\node (A1) at (5,0) {\textbf{1}};
	\node (B1) at (5,1) {\textbf{4}};
	\node (C1) at (5,2) {\textbf{2}};
	\node (D1) at (5,3) {\textbf{1}};
	\node (E1) at (5,4) {\textbf{5}};
	\node (F1) at (6,0) {\textbf{3}};
	\node (G1) at (6,1) {\textbf{5}};
	\node (H1) at (6,2) {\textbf{3}};
	\node (I1) at (6,3) {\textbf{2}};
	\node (J1) at (6,4) {\textbf{4}};
	\end{scope}
	
	\begin{scope}[>={[black]},
	every node/.style={fill=white,circle},
	every edge/.style={draw=black}]
	\path [->] (A) edge (B);
	\path [->] (B) edge (C);
	\path [->] (C) edge (D);
	\path [->] (D) edge (E);
	\path [->] (F) edge (G);
	\path [->] (G) edge (H);
	\path [->] (H) edge (I);
	\path [->] (I) edge (J);
	\path [->] (K) edge (L);
	\path [->] (L) edge (M);
	\path [->] (M) edge (N);
	\path [->] (N) edge (O);
	\path [->] (P) edge (Q);
	\path [->] (Q) edge (R);
	\path [->] (R) edge (S);
	\path [->] (S) edge (S1);
	\path [->] (T) edge (U);
	\path [->] (U) edge (V);
	\path [->] (V) edge (W);
	\path [->] (W) edge (X);
	\path [->] (A1) edge (B1);
	\path [->] (B1) edge (C1);
	\path [->] (C1) edge (D1);
	\path [->] (D1) edge (E1);
	\path [->] (F1) edge (G1);
	\path [->] (G1) edge (H1);
	\path [->] (H1) edge (I1);
	\path [->] (I1) edge (J1);
	\path [->] (A) edge (F);
	\path [->] (F) edge (K);
	\path [->] (K) edge (P);
	\path [->] (T) edge (A1);
	\path [->] (A1) edge (F1);
	\path [->] (B) edge (G);
	\path [->] (G) edge (L);
	\path [->] (L) edge (Q);
	\path [->] (U) edge (B1);
	\path [->] (B1) edge (G1);
	\path [->] (C) edge (H);
	\path [->] (H) edge (M);
	\path [->] (M) edge (R);
	\path [->] (V) edge (C1);
	\path [->] (C1) edge (H1);
	\path [->] (D) edge (I);
	\path [->] (I) edge (N);
	\path [->] (N) edge (S);
	\path [->] (W) edge (D1);
	\path [->] (D1) edge (I1);
	\path [->] (P) edge (T);
	\path [->] (Q) edge (U);
	\path [->] (R) edge (V);
	\path [->] (S) edge (W);
	\path [->] (E) edge (J);
	\path [->] (J) edge (O);
	\path [->] (O) edge (S1);
	\path [->] (S1) edge (X);
	\path [->] (X) edge (E1);
	\path [->] (E1) edge (J1);
	\path [->] (A) edge[bend left=20] (E);
	\path [->] (F) edge[bend right=20] (J);
	\path [->] (K) edge[bend right=20] (O);
	\path [->] (P) edge[bend right=20] (S1);
	\path [->] (T) edge[bend right=20] (X);
	\path [->] (A1) edge[bend right=20] (E1);
	\path [->] (F1) edge[bend right=20] (J1);
	\path [->] (A) edge[bend right=20] (F1);
	\path [->] (B) edge[bend right=20] (G1);
	\path [->] (C) edge[bend right=20] (H1);
	\path [->] (D) edge[bend right=20] (I1);
	\path [->] (E) edge[bend left=20] (J1);
	\end{scope}
	\end{tikzpicture}
\end{center}

\vspace{-.5cm}
\hspace*{2.7cm}$(iii)$ \hspace*{6.5cm}$(iv)$

\begin{center}
	\begin{tikzpicture}
	\begin{scope}[every node/.style={circle,thick,draw,scale=0.5}]
	\node (A) at (0,0) {\textbf{3}};
	\node (B) at (0,1) {\textbf{5}};
	\node (C) at (0,2) {\textbf{2}};
	\node (D) at (0,3) {\textbf{4}};
	\node (E) at (0,4) {\textbf{1}};
	\node (J) at (1,0) {\textbf{4}};
	\node (K) at (1,1) {\textbf{1}};
	\node (L) at (1,2) {\textbf{3}};
	\node (M) at (1,3) {\textbf{5}};
	\node (N) at (1,4) {\textbf{2}};
	\node (S) at (2,0) {\textbf{5}};
	\node (T) at (2,1) {\textbf{2}};
	\node (U) at (2,2) {\textbf{4}};
	\node (V) at (2,3) {\textbf{1}};
	\node (W) at (2,4) {\textbf{3}};
	\node (BB) at (3,0) {\textbf{1}};
	\node (CC) at (3,1) {\textbf{3}};
	\node (DD) at (3,2) {\textbf{5}};
	\node (EE) at (3,3) {\textbf{2}};
	\node (FF) at (3,4) {\textbf{4}};
	\node (KK) at (4,0) {\textbf{2}};
	\node (LL) at (4,1) {\textbf{4}};
	\node (MM) at (4,2) {\textbf{1}};
	\node (NN) at (4,3) {\textbf{3}};
	\node (OO) at (4,4) {\textbf{5}};
	\node (TT) at (5,0) {\textbf{3}};
	\node (UU) at (5,1) {\textbf{5}};
	\node (VV) at (5,2) {\textbf{2}};
	\node (WW) at (5,3) {\textbf{4}};
	\node (XX) at (5,4) {\textbf{1}};
	\node (C1) at (6,0) {\textbf{4}};
	\node (D1) at (6,1) {\textbf{1}};
	\node (E1) at (6,2) {\textbf{3}};
	\node (F1) at (6,3) {\textbf{5}};
	\node (G1) at (6,4) {\textbf{2}};
	\node (L1) at (7,0) {\textbf{5}};
	\node (M1) at (7,1) {\textbf{2}};
	\node (N1) at (7,2) {\textbf{4}};
	\node (O1) at (7,3) {\textbf{1}};
	\node (P1) at (7,4) {\textbf{3}};
	\node (U1) at (8,0) {\textbf{1}};
	\node (V1) at (8,1) {\textbf{3}};
	\node (W1) at (8,2) {\textbf{5}};
	\node (X1) at (8,3) {\textbf{2}};
	\node (Y1) at (8,4) {\textbf{4}};
	\node (D2) at (9,0) {\textbf{5}};
	\node (E2) at (9,1) {\textbf{2}};
	\node (F2) at (9,2) {\textbf{4}};
	\node (G2) at (9,3) {\textbf{1}};
	\node (H2) at (9,4) {\textbf{3}};
	\node (M2) at (10,0) {\textbf{4}};
	\node (N2) at (10,1) {\textbf{1}};
	\node (O2) at (10,2) {\textbf{3}};
	\node (P2) at (10,3) {\textbf{5}};
	\node (Q2) at (10,4) {\textbf{2}};
	\end{scope}
	
	\begin{scope}[>={[black]},
	every node/.style={fill=white,circle},
	every edge/.style={draw=black}]
	\path [->] (A) edge (B);
	\path [->] (B) edge (C);
	\path [->] (C) edge (D);
	\path [->] (D) edge (E);
	\path [->] (J) edge (K);
	\path [->] (K) edge (L);
	\path [->] (L) edge (M);
	\path [->] (M) edge (N);
	\path [->] (N) edge (O);
	\path [->] (S) edge (T);
	\path [->] (T) edge (U);
	\path [->] (U) edge (V);
	\path [->] (V) edge (W);
	\path [->] (BB) edge (CC);
	\path [->] (CC) edge (DD);
	\path [->] (DD) edge (EE);
	\path [->] (EE) edge (FF);
	\path [->] (KK) edge (LL);
	\path [->] (LL) edge (MM);
	\path [->] (MM) edge (NN);
	\path [->] (NN) edge (OO);
	\path [->] (TT) edge (UU);
	\path [->] (UU) edge (VV);
	\path [->] (VV) edge (WW);
	\path [->] (WW) edge (XX);
	\path [->] (C1) edge (D1);
	\path [->] (D1) edge (E1);
	\path [->] (E1) edge (F1);
	\path [->] (F1) edge (G1);
	\path [->] (L1) edge (M1);
	\path [->] (M1) edge (N1);
	\path [->] (N1) edge (O1);
	\path [->] (O1) edge (P1);
	\path [->] (U1) edge (V1);
	\path [->] (V1) edge (W1);
	\path [->] (W1) edge (X1);
	\path [->] (X1) edge (Y1);
	\path [->] (D2) edge (E2);
	\path [->] (E2) edge (F2);
	\path [->] (F2) edge (G2);
	\path [->] (G2) edge (H2);
	\path [->] (M2) edge (N2);
	\path [->] (N2) edge (O2);
	\path [->] (O2) edge (P2);
	\path [->] (P2) edge (Q2);
	\path [->] (A) edge (J);
	\path [->] (J) edge (S);
	\path [->] (S) edge (BB);
	\path [->] (BB) edge (KK);
	\path [->] (KK) edge (TT);
	\path [->] (TT) edge (C1);
	\path [->] (C1) edge (L1);
	\path [->] (L1) edge (U1);
	\path [->] (U1) edge (D2);
	\path [->] (D2) edge (M2);
	\path [->] (B) edge (K);
	\path [->] (K) edge (T);
	\path [->] (T) edge (CC);
	\path [->] (CC) edge (LL);
	\path [->] (LL) edge (UU);
	\path [->] (UU) edge (D1);
	\path [->] (D1) edge (M1);
	\path [->] (M1) edge (V1);
	\path [->] (V1) edge (E2);
	\path [->] (E2) edge (N2);
	\path [->] (C) edge (L);
	\path [->] (L) edge (U);
	\path [->] (U) edge (DD);
	\path [->] (DD) edge (MM);
	\path [->] (MM) edge (VV);
	\path [->] (VV) edge (E1);
	\path [->] (E1) edge (N1);
	\path [->] (N1) edge (W1);
	\path [->] (W1) edge (F2);
	\path [->] (F2) edge (O2);
	\path [->] (D) edge (M);
	\path [->] (M) edge (V);
	\path [->] (V) edge (EE);
	\path [->] (EE) edge (NN);
	\path [->] (NN) edge (WW);
	\path [->] (WW) edge (F1);
	\path [->] (F1) edge (O1);
	\path [->] (O1) edge (X1);
	\path [->] (X1) edge (G2);
	\path [->] (G2) edge (P2);
	\path [->] (E) edge (N);
	\path [->] (N) edge (W);
	\path [->] (W) edge (FF);
	\path [->] (FF) edge (OO);
	\path [->] (OO) edge (XX);
	\path [->] (XX) edge (G1);
	\path [->] (G1) edge (P1);
	\path [->] (P1) edge (Y1);
	\path [->] (Y1) edge (H2);
	\path [->] (H2) edge (Q2);
	\path [->] (A) edge[bend left=20] (E);
	\path [->] (J) edge[bend right=20] (N);
	\path [->] (S) edge[bend right=20] (W);
	\path [->] (BB) edge[bend right=20] (FF);
	\path [->] (KK) edge[bend right=20] (OO);
	\path [->] (TT) edge[bend right=20] (XX);
	\path [->] (C1) edge[bend right=20] (G1);
	\path [->] (L1) edge[bend right=20] (P1);
	\path [->] (U1) edge[bend right=20] (Y1);
	\path [->] (D2) edge[bend right=20] (H2);
	\path [->] (M2) edge[bend right=20] (Q2);
	\path [->] (A) edge[bend right=10] (M2);
	\path [->] (B) edge[bend right=10] (N2);
	\path [->] (C) edge[bend right=10] (O2);
	\path [->] (D) edge[bend right=10] (P2);
	\path [->] (E) edge[bend left=10] (Q2);
	\end{scope}
	\end{tikzpicture}
	
	$(v)$
\end{center}

\begin{center}
Figure $4$: $5$-star colorings of $C_5\square C_n$, for  $n=4,5,6,7,11$
\end{center}

}
\end{proof}


\begin{thm}\label{6n}
{For every two natural numbers $m$ and $n\ge 3$, $m\in \{6,8,9,10\}$, $\chi_s(C_m\square C_n)=5$.
}\end{thm}

\begin{proof} 
{By  Theorems \ref{fertin} and \ref{3n}, and using some suitable copies of $C_3\square C_n$, one can see that $\chi_s(C_6\square C_n)=\chi_s(C_9\square C_n)=5$, for every natural number $n$, $n\neq 3,5$. 
 Figure $2,\,(ii)$, Figure $4,\, (iii)$ and Figure $2,\,(iv)$ provide $5$-star colorings of $C_3\square C_6,\,C_5\square C_6$ and $C_3\square C_9$, respectively. Also, noting to Figure $4,\,(i)$ and $(ii)$, we obtain a $5$-star coloring of $C_5\square C_9$. Similarly,  by Theorems  \ref{fertin}, \ref{4n} and \ref{5n} and using two copies of $C_4\square C_n$ and two copies of $C_5\square C_n$, one can see that $\chi_s(C_8\square C_n)=\chi_s(C_{10}\square C_n)=5$ and the proof is complete.
 }
\end{proof}


\begin{thm}\label{7n}
{For every natural number $n\ge 3$, $\chi_s(C_7\square C_n)=5$.
}\end{thm}

\begin{proof} 
{By Theorems \ref{3n}, \ref{4n} and \ref{5n},  one can assume that $n\ge 6$. Theorem \ref{prime} states that every $n \ge 6$ is a linear combination of $3$ and $4$ with non-negative integer coefficients. Now, although
the colors of the first two   columns of Figure $5,\,(i)$ and $(ii)$ are the same, by considering some suitable copies of Figure $5,\, (i)$ and $(ii)$, one may still obtain a $5$-star coloring of $C_7\square C_n$, for  $n\ge 6$. So,  using Theorem \ref{fertin} completes the proof.
\begin{center}
	\begin{tikzpicture}
	\begin{scope}[every node/.style={circle,thick,draw,scale=0.5}]
	\node (A) at (0,0) {\textbf{3}};
	\node (B) at (0,1) {\textbf{2}};
	\node (C) at (0,2) {\textbf{5}};
	\node (D) at (0,3) {\textbf{1}};
	\node (E) at (0,4) {\textbf{3}};
	\node (F) at (0,5) {\textbf{2}};
	\node (F1) at (0,6) {\textbf{4}};
	\node (G) at (1,0) {\textbf{5}};
	\node (H) at (1,1) {\textbf{1}};
	\node (I) at (1,2) {\textbf{4}};
	\node (J) at (1,3) {\textbf{5}};
	\node (K) at (1,4) {\textbf{2}};
	\node (L) at (1,5) {\textbf{1}};
	\node (L1) at (1,6) {\textbf{3}};
	\node (M) at (2,0) {\textbf{2}};
	\node (N) at (2,1) {\textbf{4}};
	\node (O) at (2,2) {\textbf{3}};
	\node (P) at (2,3) {\textbf{2}};
	\node (Q) at (2,4) {\textbf{4}};
	\node (R) at (2,5) {\textbf{5}};
	\node (R1) at (2,6) {\textbf{1}};
	\end{scope}
	
	\begin{scope}[>={[black]},
	every node/.style={fill=white,circle},
	every edge/.style={draw=black}]
	\path [->] (A) edge (B);
	\path [->] (B) edge (C);
	\path [->] (C) edge (D);
	\path [->] (D) edge (E);
	\path [->] (E) edge (F);
	\path [->] (F) edge (F1);
	\path [->] (G) edge (H);
	\path [->] (H) edge (I);
	\path [->] (I) edge (J);
	\path [->] (J) edge (K);
	\path [->] (K) edge (L);
	\path [->] (L) edge (L1);
	\path [->] (M) edge (N);
	\path [->] (N) edge (O);
	\path [->] (O) edge (P);
	\path [->] (P) edge (Q);
	\path [->] (Q) edge (R);
	\path [->] (R) edge (R1);
	\path [->] (A) edge (G);
	\path [->] (G) edge (M);
	\path [->] (B) edge (H);
	\path [->] (H) edge (N);
	\path [->] (C) edge (I);
	\path [->] (I) edge (O);
	\path [->] (D) edge (J);
	\path [->] (J) edge (P);
	\path [->] (E) edge (K);
	\path [->] (K) edge (Q);
	\path [->] (F) edge (L);
	\path [->] (L) edge (R);
	\path [->] (F1) edge (L1);
	\path [->] (L1) edge (R1);
	\path [->] (A) edge[bend left=20] (F1);
	\path [->] (G) edge[bend right=20] (L1);
	\path [->] (M) edge[bend right=20] (R1);
	\path [->] (A) edge[bend right=20] (M);
	\path [->] (B) edge[bend right=20] (N);
	\path [->] (C) edge[bend right=20] (O);
	\path [->] (D) edge[bend right=20] (P);
	\path [->] (E) edge[bend right=20] (Q);
	\path [->] (F) edge[bend right=20] (R);
	\path [->] (F1) edge[bend left=20] (R1);
	\end{scope}
	\end{tikzpicture}
	\hspace{.5cm}
	\begin{tikzpicture}
	\begin{scope}[every node/.style={circle,thick,draw,scale=0.5}]
	\node (A) at (0,0) {\textbf{3}};
	\node (B) at (0,1) {\textbf{2}};
	\node (C) at (0,2) {\textbf{5}};
	\node (D) at (0,3) {\textbf{1}};
	\node (E) at (0,4) {\textbf{3}};
	\node (F) at (0,5) {\textbf{2}};
	\node (F1) at (0,6) {\textbf{4}};
	\node (G) at (1,0) {\textbf{5}};
	\node (H) at (1,1) {\textbf{1}};
	\node (I) at (1,2) {\textbf{4}};
	\node (J) at (1,3) {\textbf{5}};
	\node (K) at (1,4) {\textbf{2}};
	\node (L) at (1,5) {\textbf{1}};
	\node (L1) at (1,6) {\textbf{3}};
	\node (M) at (2,0) {\textbf{3}};
	\node (N) at (2,1) {\textbf{4}};
	\node (O) at (2,2) {\textbf{2}};
	\node (P) at (2,3) {\textbf{3}};
	\node (Q) at (2,4) {\textbf{5}};
	\node (R) at (2,5) {\textbf{4}};
	\node (R1) at (2,6) {\textbf{2}};
	\node (S) at (3,0) {\textbf{1}};
	\node (T) at (3,1) {\textbf{5}};
	\node (U) at (3,2) {\textbf{3}};
	\node (V) at (3,3) {\textbf{4}};
	\node (W) at (3,4) {\textbf{1}};
	\node (X) at (3,5) {\textbf{5}};
	\node (X1) at (3,6) {\textbf{3}};
	\end{scope}
	
	\begin{scope}[>={[black]},
	every node/.style={fill=white,circle},
	every edge/.style={draw=black}]
	\path [->] (A) edge (B);
	\path [->] (B) edge (C);
	\path [->] (C) edge (D);
	\path [->] (D) edge (E);
	\path [->] (E) edge (F);
	\path [->] (F) edge (F1);
	\path [->] (G) edge (H);
	\path [->] (H) edge (I);
	\path [->] (I) edge (J);
	\path [->] (J) edge (K);
	\path [->] (K) edge (L);
	\path [->] (L) edge (L1);
	\path [->] (M) edge (N);
	\path [->] (N) edge (O);
	\path [->] (O) edge (P);
	\path [->] (P) edge (Q);
	\path [->] (Q) edge (R);
	\path [->] (R) edge (R1);
	\path [->] (S) edge (T);
	\path [->] (T) edge (U);
	\path [->] (U) edge (V);
	\path [->] (V) edge (W);
	\path [->] (W) edge (X);
	\path [->] (X) edge (X1);
	\path [->] (A) edge (G);
	\path [->] (G) edge (M);
	\path [->] (M) edge (S);
	\path [->] (B) edge (H);
	\path [->] (H) edge (N);
	\path [->] (N) edge (T);
	\path [->] (C) edge (I);
	\path [->] (I) edge (O);
	\path [->] (O) edge (U);
	\path [->] (D) edge (J);
	\path [->] (J) edge (P);
	\path [->] (P) edge (V);
	\path [->] (E) edge (K);
	\path [->] (K) edge (Q);
	\path [->] (Q) edge (W);
	\path [->] (F) edge (L);
	\path [->] (L) edge (R);
	\path [->] (R) edge (X);
	\path [->] (F1) edge (L1);
	\path [->] (L1) edge (R1);
	\path [->] (R1) edge (X1);
	\path [->] (A) edge[bend left=20] (F1);
	\path [->] (G) edge[bend right=20] (L1);
	\path [->] (M) edge[bend right=20] (R1);
	\path [->] (S) edge[bend right=20] (X1);
	\path [->] (A) edge[bend right=20] (S);
	\path [->] (B) edge[bend right=20] (T);
	\path [->] (C) edge[bend right=20] (U);
	\path [->] (D) edge[bend right=20] (V);
	\path [->] (E) edge[bend right=20] (W);
	\path [->] (F) edge[bend right=20] (X);
	\path [->] (F1) edge[bend left=20] (X1);
	\end{scope}
	\end{tikzpicture}
\end{center}

\vspace{-.5cm}
\hspace*{4cm}$(i)$ \hspace*{4.3cm}$(ii)$
\begin{center}
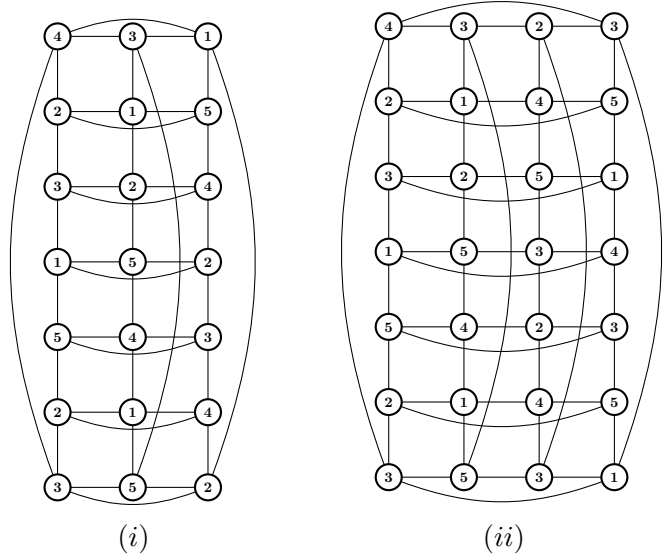
Figure $5$:  $5$-star colorings of $C_7\square C_n$, for $n=3,4$
\end{center}

}  
\end{proof}


\begin{thm}\label{11n}
{For every natural number $n\ge 3$, $\chi_s(C_{11}\square C_n)=5$.
}\end{thm}

\begin{proof} 
{Clearly, by Theorems \ref{3n},$\ldots$,\ref{7n},  one can assume that $n\ge 11$. Theorem \ref{prime} states that every $n \ge 12$ is a linear combination of $4$ and $5$ with non-negative integer coefficients. Now, since 
the colors of the first three columns of Figures $3,\,(v)$ and $4,\,(v)$ are the same, by considering some suitable copies of Figures $3,\, (v)$ and $4\, (v)$, one may obtain a $5$-star coloring of $C_{11}\square C_n$, for $n\ge 12$. For $n=11$, consider the following $5$-star coloring. Now, by Theorem \ref{fertin}, the proof is complete.

\begin{center}
	\begin{tikzpicture}
	\begin{scope}[every node/.style={circle,thick,draw,scale=0.5}]
	\node (A) at (0,0) {\textbf{3}};
	\node (B) at (0,1) {\textbf{2}};
	\node (C) at (0,2) {\textbf{4}};
	\node (D) at (0,3) {\textbf{5}};
	\node (E) at (0,4) {\textbf{1}};
	\node (F) at (0,5) {\textbf{2}};
	\node (F1) at (0,6) {\textbf{4}};
	\node (F2) at (0,7) {\textbf{3}};
	\node (FF3) at (0,8) {\textbf{2}};
	\node (FF4) at (0,9) {\textbf{4}};
	\node (FF5) at (0,10) {\textbf{5}};
	\node (G) at (1,0) {\textbf{4}};
	\node (H) at (1,1) {\textbf{5}};
	\node (I) at (1,2) {\textbf{2}};
	\node (J) at (1,3) {\textbf{1}};
	\node (K) at (1,4) {\textbf{3}};
	\node (L) at (1,5) {\textbf{5}};
	\node (L1) at (1,6) {\textbf{1}};
	\node (L2) at (1,7) {\textbf{4}};
	\node (LL3) at (1,8) {\textbf{5}};
	\node (LL4) at (1,9) {\textbf{2}};
	\node (LL5) at (1,10) {\textbf{1}};
	\node (M) at (2,0) {\textbf{5}};
	\node (N) at (2,1) {\textbf{1}};
	\node (O) at (2,2) {\textbf{3}};
	\node (P) at (2,3) {\textbf{2}};
	\node (Q) at (2,4) {\textbf{4}};
	\node (R) at (2,5) {\textbf{3}};
	\node (R1) at (2,6) {\textbf{2}};
	\node (R2) at (2,7) {\textbf{5}};
	\node (R3) at (2,8) {\textbf{1}};
	\node (R4) at (2,9) {\textbf{3}};
	\node (R5) at (2,10) {\textbf{2}};
	\node (S) at (3,0) {\textbf{3}};
	\node (T) at (3,1) {\textbf{4}};
	\node (U) at (3,2) {\textbf{1}};
	\node (V) at (3,3) {\textbf{5}};
	\node (W) at (3,4) {\textbf{3}};
	\node (X) at (3,5) {\textbf{1}};
	\node (X1) at (3,6) {\textbf{5}};
	\node (X2) at (3,7) {\textbf{3}};
	\node (X3) at (3,8) {\textbf{4}};
	\node (X4) at (3,9) {\textbf{1}};
	\node (X5) at (3,10) {\textbf{5}};
	\node (Y) at (4,0) {\textbf{2}};
	\node (Z) at (4,1) {\textbf{3}};
	\node (AA) at (4,2) {\textbf{5}};
	\node (BB) at (4,3) {\textbf{4}};
	\node (CC) at (4,4) {\textbf{2}};
	\node (DD) at (4,5) {\textbf{3}};
	\node (D1) at (4,6) {\textbf{4}};
	\node (D2) at (4,7) {\textbf{2}};
	\node (D3) at (4,8) {\textbf{3}};
	\node (D4) at (4,9) {\textbf{5}};
	\node (D5) at (4,10) {\textbf{4}};
	\node (EE) at (5,0) {\textbf{4}};
	\node (FF) at (5,1) {\textbf{5}};
	\node (GG) at (5,2) {\textbf{2}};
	\node (HH) at (5,3) {\textbf{1}};
	\node (II) at (5,4) {\textbf{3}};
	\node (JJ) at (5,5) {\textbf{5}};
	\node (J1) at (5,6) {\textbf{1}};
	\node (J2) at (5,7) {\textbf{4}};
	\node (JJ3) at (5,8) {\textbf{5}};
	\node (JJ4) at (5,9) {\textbf{2}};
	\node (JJ5) at (5,10) {\textbf{1}};
	\node (E3) at (6,0) {\textbf{5}};
	\node (F3) at (6,1) {\textbf{1}};
	\node (G3) at (6,2) {\textbf{3}};
	\node (H3) at (6,3) {\textbf{2}};
	\node (I3) at (6,4) {\textbf{4}};
	\node (J3) at (6,5) {\textbf{3}};
	\node (K3) at (6,6) {\textbf{2}};
	\node (L3) at (6,7) {\textbf{5}};
	\node (MM3) at (6,8) {\textbf{1}};
	\node (MM4) at (6,9) {\textbf{3}};
	\node (MM5) at (6,10) {\textbf{2}};
	\node (E4) at (7,0) {\textbf{3}};
	\node (F4) at (7,1) {\textbf{4}};
	\node (G4) at (7,2) {\textbf{1}};
	\node (H4) at (7,3) {\textbf{5}};
	\node (I4) at (7,4) {\textbf{3}};
	\node (J4) at (7,5) {\textbf{1}};
	\node (K4) at (7,6) {\textbf{5}};
	\node (L4) at (7,7) {\textbf{3}};
	\node (M4) at (7,8) {\textbf{4}};
	\node (N4) at (7,9) {\textbf{1}};
	\node (O4) at (7,10) {\textbf{5}};
	\node (AB) at (8,0) {\textbf{2}};
	\node (CD) at (8,1) {\textbf{3}};
	\node (EF) at (8,2) {\textbf{5}};
	\node (GH) at (8,3) {\textbf{4}};
	\node (IJ) at (8,4) {\textbf{2}};
	\node (KL) at (8,5) {\textbf{3}};
	\node (MN) at (8,6) {\textbf{4}};
	\node (OP) at (8,7) {\textbf{2}};
	\node (QR) at (8,8) {\textbf{3}};
	\node (ST) at (8,9) {\textbf{5}};
	\node (UV) at (8,10) {\textbf{4}};
	\node (WX) at (9,0) {\textbf{4}};
	\node (YZ) at (9,1) {\textbf{5}};
	\node (BA) at (9,2) {\textbf{2}};
	\node (DC) at (9,3) {\textbf{1}};
	\node (FE) at (9,4) {\textbf{3}};
	\node (HG) at (9,5) {\textbf{5}};
	\node (JI) at (9,6) {\textbf{1}};
	\node (LK) at (9,7) {\textbf{4}};
	\node (NM) at (9,8) {\textbf{5}};
	\node (PO) at (9,9) {\textbf{2}};
	\node (RQ) at (9,10) {\textbf{1}};
	\node (TS) at (10,0) {\textbf{5}};
	\node (VU) at (10,1) {\textbf{1}};
	\node (XW) at (10,2) {\textbf{3}};
	\node (ZY) at (10,3) {\textbf{2}};
	\node (AAA) at (10,4) {\textbf{4}};
	\node (BBB) at (10,5) {\textbf{3}};
	\node (CCC) at (10,6) {\textbf{2}};
	\node (DDD) at (10,7) {\textbf{5}};
	\node (EEE) at (10,8) {\textbf{1}};
	\node (FFF) at (10,9) {\textbf{3}};
	\node (GGG) at (10,10) {\textbf{2}};
	\end{scope}
	
	\begin{scope}[>={[black]},
	every node/.style={fill=white,circle},
	every edge/.style={draw=black}]
	\path [->] (A) edge (B);
	\path [->] (B) edge (C);
	\path [->] (C) edge (D);
	\path [->] (D) edge (E);
	\path [->] (E) edge (F);
	\path [->] (F) edge (F1);
	\path [->] (F1) edge (F2);
	\path [->] (F2) edge (FF3);
	\path [->] (FF3) edge (FF4);
	\path [->] (FF4) edge (FF5);
	\path [->] (G) edge (H);
	\path [->] (H) edge (I);
	\path [->] (I) edge (J);
	\path [->] (J) edge (K);
	\path [->] (K) edge (L);
	\path [->] (L) edge (L1);
	\path [->] (L1) edge (L2);
	\path [->] (L2) edge (LL3);
	\path [->] (LL3) edge (LL4);
	\path [->] (LL4) edge (LL5);
	\path [->] (M) edge (N);
	\path [->] (N) edge (O);
	\path [->] (O) edge (P);
	\path [->] (P) edge (Q);
	\path [->] (Q) edge (R);
	\path [->] (R) edge (R1);
	\path [->] (R1) edge (R2);
	\path [->] (S) edge (T);
	\path [->] (T) edge (U);
	\path [->] (U) edge (V);
	\path [->] (V) edge (W);
	\path [->] (W) edge (X);
	\path [->] (X) edge (X1);
	\path [->] (X1) edge (X2);
	\path [->] (Y) edge (Z);
	\path [->] (Z) edge (AA);
	\path [->] (AA) edge (BB);
	\path [->] (BB) edge (CC);
	\path [->] (CC) edge (DD);
	\path [->] (DD) edge (D1);
	\path [->] (D1) edge (D2);
	\path [->] (EE) edge (FF);
	\path [->] (FF) edge (GG);
	\path [->] (GG) edge (HH);
	\path [->] (HH) edge (II);
	\path [->] (II) edge (JJ);
	\path [->] (JJ) edge (J1);
	\path [->] (J1) edge (J2);
	\path [->] (F) edge (L);
	\path [->] (L) edge (R);
	\path [->] (R) edge (X);
	\path [->] (X) edge (DD);
	\path [->] (DD) edge (JJ);
	\path [->] (E3) edge (F3);
	\path [->] (F3) edge (G3);
	\path [->] (G3) edge (H3);
	\path [->] (H3) edge (I3);
	\path [->] (I3) edge (J3);
	\path [->] (J3) edge (K3);
	\path [->] (K3) edge (L3);
	\path [->] (E4) edge (F4);
	\path [->] (F4) edge (G4);
	\path [->] (G4) edge (H4);
	\path [->] (H4) edge (I4);
	\path [->] (I4) edge (J4);
	\path [->] (J4) edge (K4);
	\path [->] (K4) edge (L4);
	\path [->] (E) edge (K);
	\path [->] (K) edge (Q);
	\path [->] (Q) edge (W);
	\path [->] (W) edge (CC);
	\path [->] (CC) edge (II);
	\path [->] (D) edge (J);
	\path [->] (J) edge (P);
	\path [->] (P) edge (V);
	\path [->] (V) edge (BB);
	\path [->] (BB) edge (HH);
	\path [->] (C) edge (I);
	\path [->] (I) edge (O);
	\path [->] (O) edge (U);
	\path [->] (U) edge (AA);
	\path [->] (AA) edge (GG);
	\path [->] (B) edge (H);
	\path [->] (H) edge (N);
	\path [->] (N) edge (T);
	\path [->] (T) edge (Z);
	\path [->] (Z) edge (FF);
	\path [->] (A) edge (G);
	\path [->] (G) edge (M);
	\path [->] (M) edge (S);
	\path [->] (S) edge (Y);
	\path [->] (Y) edge (EE);
	\path [->] (F1) edge (L1);
	\path [->] (L1) edge (R1);
	\path [->] (R1) edge (X1);
	\path [->] (X1) edge (D1);
	\path [->] (D1) edge (J1);
	\path [->] (J1) edge (K3);
	\path [->] (K3) edge (K4);
	\path [->] (F2) edge (L2);
	\path [->] (L2) edge (R2);
	\path [->] (R2) edge (X2);
	\path [->] (X2) edge (D2);
	\path [->] (D2) edge (J2);
	\path [->] (J2) edge (L3);
	\path [->] (L3) edge (L4);
	\path [->] (EE) edge (E3);
	\path [->] (E3) edge (E4);
	\path [->] (FF) edge (F3);
	\path [->] (F3) edge (F4);
	\path [->] (GG) edge (G3);
	\path [->] (G3) edge (G4);
	\path [->] (HH) edge (H3);
	\path [->] (H3) edge (H4);
	\path [->] (II) edge (I3);
	\path [->] (I3) edge (I4);
	\path [->] (JJ) edge (J3);
	\path [->] (J3) edge (J4);
	\path [->] (R2) edge (R3);
	\path [->] (R3) edge (R4);
	\path [->] (R4) edge (R5);
	\path [->] (X2) edge (X3);
	\path [->] (X3) edge (X4);
	\path [->] (X4) edge (X5);
	\path [->] (D2) edge (D3);
	\path [->] (D3) edge (D4);
	\path [->] (D4) edge (D5);
	\path [->] (J2) edge (JJ3);
	\path [->] (JJ3) edge (JJ4);
	\path [->] (JJ4) edge (JJ5);
	\path [->] (L3) edge (MM3);
	\path [->] (MM3) edge (MM4);
	\path [->] (MM4) edge (MM5);
	\path [->] (L4) edge (M4);
	\path [->] (M4) edge (N4);
	\path [->] (N4) edge (O4);
	\path [->] (OP) edge (QR);
	\path [->] (QR) edge (ST);
	\path [->] (ST) edge (UV);
	\path [->] (LK) edge (NM);
	\path [->] (NM) edge (PO);
	\path [->] (PO) edge (RQ);
	\path [->] (DDD) edge (EEE);
	\path [->] (EEE) edge (FFF);
	\path [->] (FFF) edge (GGG);
	\path [->] (AB) edge (CD);
	\path [->] (CD) edge (EF);
	\path [->] (EF) edge (GH);
	\path [->] (GH) edge (IJ);
	\path [->] (IJ) edge (KL);
	\path [->] (KL) edge (MN);
	\path [->] (OP) edge (MN);
	\path [->] (WX) edge (YZ);
	\path [->] (YZ) edge (BA);
	\path [->] (BA) edge (DC);
	\path [->] (DC) edge (FE);
	\path [->] (FE) edge (HG);
	\path [->] (HG) edge (JI);
	\path [->] (JI) edge (LK);
	\path [->] (TS) edge (VU);
	\path [->] (VU) edge (XW);
	\path [->] (XW) edge (ZY);
	\path [->] (ZY) edge (AAA);
	\path [->] (AAA) edge (BBB);
	\path [->] (BBB) edge (CCC);
	\path [->] (CCC) edge (DDD);
	\path [->] (FF3) edge (LL3);
	\path [->] (LL3) edge (R3);
	\path [->] (R3) edge (X3);
	\path [->] (X3) edge (D3);
	\path [->] (D3) edge (JJ3);
	\path [->] (JJ3) edge (MM3);
	\path [->] (MM3) edge (M4);
	\path [->] (M4) edge (QR);
	\path [->] (QR) edge (NM);
	\path [->] (NM) edge (EEE);
	\path [->] (FF4) edge (LL4);
	\path [->] (LL4) edge (R4);
	\path [->] (R4) edge (X4);
	\path [->] (X4) edge (D4);
	\path [->] (D4) edge (JJ4);
	\path [->] (JJ4) edge (MM4);
	\path [->] (MM4) edge (N4);
	\path [->] (N4) edge (ST);
	\path [->] (ST) edge (PO);
	\path [->] (PO) edge (FFF);
	\path [->] (FF5) edge (LL5);
	\path [->] (LL5) edge (R5);
	\path [->] (R5) edge (X5);
	\path [->] (X5) edge (D5);
	\path [->] (D5) edge (JJ5);
	\path [->] (JJ5) edge (MM5);
	\path [->] (MM5) edge (O4);
	\path [->] (O4) edge (UV);
	\path [->] (UV) edge (RQ);
	\path [->] (RQ) edge (GGG);
	\path [->] (L4) edge (OP);
	\path [->] (OP) edge (LK);
	\path [->] (LK) edge (DDD);
	\path [->] (K4) edge (MN);
	\path [->] (MN) edge (JI);
	\path [->] (JI) edge (CCC);
	\path [->] (J4) edge (KL);
	\path [->] (I4) edge (IJ);
	\path [->] (H4) edge (GH);
	\path [->] (G4) edge (EF);
	\path [->] (F4) edge (CD);
	\path [->] (E4) edge (AB);
	\path [->] (AB) edge (WX);
	\path [->] (CD) edge (YZ);
	\path [->] (EF) edge (BA);
	\path [->] (GH) edge (DC);
	\path [->] (IJ) edge (FE);
	\path [->] (KL) edge (HG);
	\path [->] (WX) edge (TS);
	\path [->] (YZ) edge (VU);
	\path [->] (BA) edge (XW);
	\path [->] (DC) edge (ZY);
	\path [->] (FE) edge (AAA);
	\path [->] (HG) edge (BBB);
	\path [->] (A) edge[bend left=10] (FF5);
	\path [->] (G) edge[bend right=10] (LL5);
	\path [->] (M) edge[bend right=10] (R5);
	\path [->] (S) edge[bend right=10] (X5);
	\path [->] (Y) edge[bend right=10] (D5);
	\path [->] (EE) edge[bend right=10] (JJ5);
	\path [->] (E3) edge[bend right=10] (MM5);
	\path [->] (E4) edge[bend right=10] (O4);
	\path [->] (AB) edge[bend right=10] (UV);
	\path [->] (WX) edge[bend right=10] (RQ);
	\path [->] (TS) edge[bend right=10] (GGG);
	\path [->] (A) edge[bend right=10] (TS);
	\path [->] (B) edge[bend right=10] (VU);
	\path [->] (C) edge[bend right=10] (XW);
	\path [->] (D) edge[bend right=10] (ZY);
	\path [->] (E) edge[bend right=10] (AAA);
	\path [->] (F) edge[bend right=10] (BBB);
	\path [->] (F1) edge[bend right=10] (CCC);
	\path [->] (F2) edge[bend right=10] (DDD);
	\path [->] (FF3) edge[bend right=10] (EEE);
	\path [->] (FF4) edge[bend right=10] (FFF);
	\path [->] (FF5) edge[bend left=10] (GGG);
	\end{scope}
	\end{tikzpicture}

\end{center}

\begin{center}
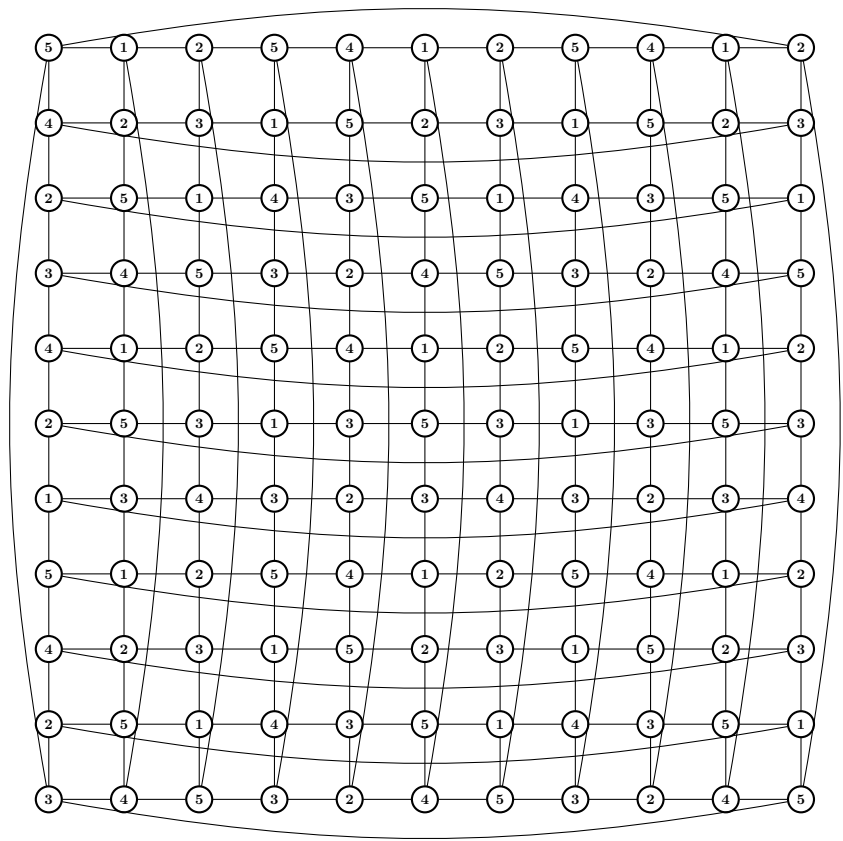
Figure $6$: A $5$-star coloring of $C_{11}\square C_{11}$
\end{center}

}  
\end{proof}


We close the paper with the following  theorem. 

\begin{thm}\label{mn}
{For natural numbers $m,n\ge 3$, $\chi_s(C_m\square C_n)=5$, except $C_3\square C_3$ and $C_3\square C_5$.
}\end{thm}

\begin{proof} 
{If $3\le m,n \le 11$, then by Theorems \ref{3n},$\ldots$,\ref{11n}, we are done. Now, by Theorem \ref{prime}, every $n \ge 12$ is a linear combination of $4$ and $5$ with non-negative integer coefficients. Since 
the colors of the first three columns of Figure $3,\,(i)$ and $(ii)$ and also  the colors of the first three columns of Figure $4,\,(i)$ and $(ii)$  are the same, by considering some suitable copies of 
these four graphs and using Theorem \ref{fertin}, the proof is complete. }  
\end{proof}

\bibliographystyle{amsplain}
%

\end{document}